\documentclass[12pt,a4paper,final]{amsart} 

\usepackage{amsmath,amsfonts,amssymb,amsthm} 
\usepackage[final]{graphicx} 
\usepackage[notref,notcite]{showkeys} 
\usepackage{comment}
\usepackage[left=3cm,right=3cm,top=2cm,bottom=3cm]{geometry}
\usepackage{mathtools}
\usepackage{paralist}
\usepackage{fancyhdr}
\usepackage{tikz}
\usepackage{libertine} 
\usepackage{libertinust1math}
\usepackage[T1]{fontenc}
\usepackage{float}
\usepackage{url} 
\usepackage{hyperref}
\usepackage[obeyDraft]{todonotes}

\theoremstyle{plain}
\newtheorem{thm}{Theorem}
\newtheorem{prop}[thm]{Proposition}
\newtheorem{lemma}[thm]{Lemma}
\newtheorem{cor}[thm]{Corollary}
\theoremstyle{definition}
\newtheorem{defi}[thm]{Definition}
\newtheorem{rem}[thm]{Remark}
\newtheorem{ex}[thm]{Example}
\newtheorem*{claim}{Claim}

\newcommand{\N}{\mathbb N} 
\newcommand{\Z}{\mathbb Z} 
\newcommand{\R}{\mathbb R} 
\newcommand{\T}{\mathbb T} 

\DeclareMathOperator{\Image}{Image}
\DeclareMathOperator{\Flux}{Flux}
\DeclareMathOperator{\PSS}{PSS}
\DeclareMathOperator{\Ham}{Ham}
\DeclareMathOperator{\conv}{conv}
\DeclareMathOperator{\Hess}{Hess}
\DeclareMathOperator{\supp}{Supp}

\DeclareMathOperator{\sh}{Sh}
\DeclareMathOperator{\Symp}{Symp}
\DeclareMathOperator{\graph}{Graph}
\DeclareMathOperator{\id}{id} 

\title{Mather theory and symplectic rigidity}
\date{}
\author{Mads R. Bisgaard}
\address{Mads Bisgaard, D-MATH, ETH Z{\"u}rich, R{\"a}mistrasse 101, 8092 Z{\"u}rich, Switzerland}
\email{madsbisgaard88@hotmail.com}
\begin{document}
\maketitle

\begin{abstract}
Using methods from symplectic topology, we prove existence of invariant variational measures associated to the flow $\phi_H$ of a Hamiltonian $H\in C^{\infty}(M)$ on a symplectic manifold $(M,\omega)$. These measures coincide with Mather measures (from Aubry-Mather theory) in the Tonelli case. We compare properties of the supports of these measures to classical Mather measures and we construct an example showing that their support can be extremely unstable when $H$ fails to be convex, even for nearly integrable $H$. Parts of these results extend work by Viterbo \cite{Viterbo10}, and Vichery \cite{Vichery14}. 

Using ideas due to Entov-Polterovich \cite{EntovPolterovich17}, \cite{Polterovich14} we also detect interesting invariant measures for $\phi_H$ by studying a generalization of the \emph{symplectic shape} of sublevel sets of $H$. This approach differs from the first one in that it works also for $(M,\omega)$ in which every compact subset can be displaced. We present applications to Hamiltonian systems on $\R^{2n}$ and twisted cotangent bundles.
\end{abstract}

\keywords{Key words: Spectral invariants, Hamiltonian systems, Invariant measures}

\subjclass{2010 AMS Subject Classification: 37J05, 37J25, 37J40, 37J50, 37D40, 37D12}

\section{Introduction}
Consider a symplectic manifold $(M^{2n},\omega)$ and a Hamiltonian $H\in C^{\infty}(M)$ with complete flow $\phi_H=\{\phi_H^t\}_{t\in \R}$. This paper concerns the question of how to systematically determine interesting invariant sets for $\phi_H$. The celebrated KAM theorem asserts (loosely speaking) that, if $H$ is non-degenerate and sufficiently regular, an invariant Lagrangian torus, on which the dynamics of $\phi_H$ has a Diophantine rotation vector, will persist under small (sufficiently regular) perturbations of $H$. Mather \cite{Mather91} discovered that one can loosen the regularity assumptions and the Diophantine condition in the KAM theorem and still find plenty of interesting invariant sets which persist perturbations of $H$, if one pays the price of replacing the non-degeneracy condition with a stronger \emph{convexity} assumption. One of the aims of the current paper is to study, using methods coming from symplectic topology, what happens when one relaxes the convexity assumption. It is interesting to understand the connection between Aubry-Mather theory and symplectic topology. Here we hope to provide some understanding of how pseudo-holomorphic curve techniques can explain phenomena from Aubry-Mather theory. A first indication that these theories can be connected was obtained by Bernard \cite{Bernard07}. The approach we take here builds heavily on Viterbo's symplectic homogenization paper \cite{Viterbo08} as well as the Floer-homological version studied by Monzner-Vichery-Zapolsky \cite{ZapolskyMonznerVichery12} and Vichery \cite{Vichery14}. The second part of the paper obtains interesting invariant measures by studying a generalization of the notion of symplectic shape (due to Sikorav \cite{Sikorav89}) of sublevel sets of a Hamiltonian. It employs ideas due to Buhovsky-Entov-Polterovich \cite{BuhovskyEntovPolterovich12}, Entov-Polterovich \cite{EntovPolterovich17} and Polterovich \cite{Polterovich14}. 

The contents of the paper are as follows: In Section \ref{MatherFloer} we present our results on existence of invariant measures using homogenized Lagrangian spectral invariants. Section \ref{Mathersec} compares properties of these measures to Mather measures. Section \ref{secC0} contains our results on invariant measures using the $C^0$-techniques developed by Buhovsky-Entov-Polterovich. Finally, Section \ref{seqprem} and \ref{secproof} contain preliminaries and proofs.        

\vspace*{0.3cm}
\noindent
\textit{Setting and notation:} $\mathcal{M}(X)$ will denote the space of Borel probability measures which are \emph{compactly} supported on $X\subset M$. We use the convention that the symplectic gradient $X_{H_t}$ associated to $H\in C^{\infty}(\R \times M)$ is defined by $\iota_{X_{H_t}}\omega=-dH_t$. The flow generated by $X_{H_t}$ is denoted by $\phi_H=\{\phi_H^t\}_{t\in \R}$. We denote by $\mathcal{M}(X;\phi_H)\subset \mathcal{M}(X)$ the subset of $\phi_H$-invariant measures and write $\mathcal{M}=\mathcal{M}(M)$ as well as $\mathcal{M}(\phi_H)=\mathcal{M}(M;\phi_H)$.  

\section{Symplectic "Mather-Floer theory"}
\label{MatherFloer}
Throughout this section $(M^{2n},\omega)$ will denote a closed monotone symplectic manifold.\footnote{In order to avoid certain compactness issues, $M$ will be assumed compact in this section. See Section \ref{noncompact} for the non-compact (exact) case.} We consider a closed \emph{monotone} Lagrangian submanifold $L\subset (M,\omega)$ which is non-narrow in the sense that its quantum homology $QH_*(L;\Z_2)$ doesn't vanish (see Section \ref{seqprem}). In this setting Leclercq-Zapolsky \cite{LeclercqZapolsky15} recently developed a theory of \emph{Lagrangian spectral invariants}. We will denote by  
\[
l_L:\widetilde{\Ham}(M,\omega) \to \R.
\]
the Leclercq-Zapolsky spectral invariant associated to the unity in $QH_*(L;\Z_2)$.\footnote{In \cite{LeclercqZapolsky15} our $l_L$ goes under the name $l_+$.} An idea due to Viterbo \cite{Viterbo08} says that homogenizing $l_L$ gives rise to an analogue of Mather's $\alpha$-function. The approach we consider here was first studied by Monzner-Vichery-Zapolsky \cite{ZapolskyMonznerVichery12} and Vichery \cite{Vichery14}. Denote by $\mathcal{H}:=\{H\in C^{\infty}(M)\ |\ \int H \omega^n=0\}$ the space of \emph{normalized} Hamiltonians.

\begin{defi}
\label{def2}
Associated to $H\in \mathcal{H}$ we define a function $\sigma_{H:L}:H^1(M;\R) \to \R$ by
\begin{equation}
\label{eq2}
\sigma_{H:L}(c)=\lim_{\N\ni k\to \infty}\frac{l_L(\psi_1^{-1}\tilde{\phi}^k_H \psi_1)}{k},
\end{equation}
where $\psi:[0,1]\times M \to M$ is any smooth symplectic isotopy with $\psi_0=\id$ and $\Flux(\psi)=c\in H^1(M;\R)$, and $\tilde{\phi}_H^k\in \widetilde{\Ham}(M,\omega)$ denotes the element represented by the path $[0,1]\ni t\mapsto \phi_H^{tk}$ in $\Ham(M,\omega)$.
\end{defi}
For more details on the construction and properties of $\sigma_{H:L}$ we refer to Section \ref{seqprem} (see in particular Theorem \ref{prop1} on page \pageref{prop1}). It turns out that $\sigma_{H:L}$ is locally Lipschitz, so at every point $c\in H^1(M;\R)$ it has a well-defined non-empty set of \emph{Clarke subdifferentials} (see Section \ref{secmeas})
\[
\partial \sigma_{H:L}(c)\subset H_1(M;\R)=H^1(M;\R)^*.
\]
The dynamical information contained in $\sigma_{H:L}$ is that it guarantees existence of analogues of \emph{Mather measures} (from Aubry-Mather theory) to the present setting, i. e. $\mathcal{M}(\phi_H)$-measures whose rotation vector (or asymptotic cycle, see Section \ref{secmeas}) is prescribed by $\partial \sigma_{H:L}(c)$. 
\begin{thm}
\label{thm1}
Let $H\in \mathcal{H}$. For any $c\in H^1(M;\R)$ and any Clarke subdifferential $h\in \partial \sigma_{H:L}(c)\subset H_1(M;\R)$ of $\sigma_{H:L}$ at $c$, there exists a $\mu \in \mathcal{M}(\phi_H)$ which has rotation vector 
\[
\rho(\mu)=h.
\] 
\end{thm}
For a more detailed discussion of how Theorem \ref{thm1} relates to classical Aubry-Mather theory we refer to Section \ref{Mathersec}.

\begin{rem}
After having proved the main part of Theorem \ref{thm1} we learned about Vichery's \cite{Vichery14}, where a result similar to Theorem \ref{thm1} is proved in the case where $L$ is the zero-section of a cotangent bundle as well as Viterbo's \cite{Viterbo10} where results from \cite{Viterbo08} are applied to yield a result similar to Theorem \ref{thm1} for $T^*\mathbb{T}^n$. The main difference between our result and Vichery's is greater generality. Moreover, our approach detects the measures directly using pseudo-holomorphic curve techniques and avoids the use of generating functions.
\end{rem}

\begin{rem}
\label{rem101}
Fix $c=0$ and consider a measure $\mu \in \mathcal{M}(\phi_H)$ with $\rho(\mu)\in \partial \sigma_{H:L}(0)$ whose existence is guaranteed by Theorem \ref{thm1}. As the proof of Theorem \ref{thm1} will show, $\mu$ arises as a convex combination of weak$^*$-limits of sequences $(\mu_{\varsigma})_{\varsigma \in \N}\subset \mathcal{M}$ characterized by 
\begin{equation}
\label{equa12}
\int f\ d\mu_{\varsigma}=\frac{1}{k_{\varsigma}}\int_0^{k_{\varsigma}}f\phi_H^t(x_{\varsigma})\ dt \quad \forall \ f\in C^0_{bd}(M).
\end{equation}
Here, $(x_{\varsigma})_{\varsigma \in \N}$ is a sequence of points with $x_{\varsigma}\in \psi_{\varsigma}(L)$ for a sequence of symplectomorphisms $(\psi_{\varsigma})_{\varsigma \in \N}\subset \Symp(M,\omega)$ satisfying
\[
\psi_{\varsigma}\stackrel{\varsigma \to \infty}{\longrightarrow}\id
\]
in the Whitney $C^{\infty}$-topology, and $(k_{\varsigma})_{\varsigma \in \N}\subset (0,\infty)$ is a sequence of times with $k_{\varsigma}\uparrow \infty$. In addition, $\phi_H^{k_{\varsigma}}(x_{\varsigma})\in \psi_{\varsigma}(L)$ for all $\varsigma \in \N$ and to each orbit $\gamma_{\varsigma}=\{\phi_H^t(x_{\varsigma})\}_{t\in [0,k_{\varsigma}]}$ is associated an action $\mathcal{A}_{H:\psi_{\varsigma}(L)}(\gamma_{\varsigma})\in \R$ such that\footnote{See Section \ref{seqprem} for a precise definition of $\mathcal{A}_{H:\psi_{\varsigma}(L)}(\gamma_{\varsigma})$.} 
\[
\frac{\mathcal{A}_{H:\psi_{\varsigma}(L)}(\gamma_{\varsigma})}{k_{\varsigma}} \stackrel{\varsigma \to \infty}{\longrightarrow}\sigma_{H:L}(0).
\]

We will denote by $\mathfrak{M}_{H:L}\subset \mathcal{M}(\phi_H)$ the set of measures $\mu \in \mathcal{M}(\phi_H)$ with $\rho(\mu)\in \partial \sigma_{H:L}(0)$ arising as convex combinations of weak$^*$-limits of sequences $(\mu_{\varsigma})_{\varsigma}$ given by (\ref{equa12}), such that $(x_{\varsigma})_{\varsigma}$ and $(k_{\varsigma})_{\varsigma}$ satisfy the criteria mentioned above. This notation is justified by results, first due to Viterbo \cite[Proposition 13.3]{Viterbo08} and later Monzner-Vichery-Zapolsky \cite[Theorem 1.11]{ZapolskyMonznerVichery12}, saying that $\sigma_{H:L}$ coincides with Mather's $\alpha$-function when $L\subset T^*L$ is the zero-section and $H$ is Tonelli. In particular, $\mathfrak{M}_{H:L}$-measures are Mather measures in this setting (see Remark \ref{Matrem901} below). For a discussion about the extent to which $\mathfrak{M}_{H:L}$ coincides with the set of all Mather measures associated to $0\in H^1(M;\R)$ we refer to Remark \ref{Matrem902} below.
\end{rem}

Under certain circumstances, one can reduce the domain of $\sigma_{H:L}$ further, thus getting further restrictions on its subdifferentials. In the following we will denote by $r_L:H^1(M;\R) \to H^1(L;\R)$ and by $i_L:H_1(L;\R) \to H_1(M;\R)$ the (co)homological maps induced by the inclusion $L\hookrightarrow M$.

\begin{cor}
\label{lem1}
Suppose $r_L$ is surjective. Then $\sigma_{H:L}$ descends to a locally Lipschitz function
\[
\alpha_{H:L}: H^1(L;\R) \to \R .
\]
In particular, for every $c\in H^1(L;\R)$ and every subdifferential $h\in \partial \alpha_{H:L}(c)\subset H_1(L;\R)$ there exists a $\mu \in \mathcal{M}(\phi_H)$ with rotation vector
\[
\rho(\mu)=i_L(h)\in H_1(M;\R).
\]
\end{cor}
Of course, the notation is meant to suggest that $\alpha_{H:L}$ should play the role of Mather's $\alpha$-function on a general symplectic manifold. 

Polterovich \cite{Polterovich14} used Poisson bracket invariants to study invariant measures which have "large" rotation vectors. In certain situations Corollary \ref{lem1} allows us to sharpen his result in the sense that we can detect that the rotation vector of the constructed measure is contained in the subspace $i_L(H_1(L;\R)) \leq H_1(M;\R)$. 

A smooth map $\psi:[0,1]\times L\to M$ is said to be a \emph{Lagrange isotopy} if each $\psi_t:L\to (M,\omega)$ is a Lagrange embedding. We use the terminology that $\psi$ \emph{starts at $L$} if $\psi_0:L\to M$ is the inclusion $L\hookrightarrow M$.
To a Lagrange isotopy $\psi$ starting at $L$ is associated a class $\Flux_L(\psi)\in H^1(L;\R)$ as follows: A smooth loop $\gamma:S^1\to L$ gives rise to a cylinder $\Gamma:S^1 \times [0,1]\to M$ via $\Gamma(s,t)=\psi_t(\gamma(s))$ and $\Flux_L(\psi)$ is characterised by
\begin{equation}
\label{Meq4}
\langle \Flux_L(\psi),[\gamma]\rangle=\int_{S^1 \times [0,1]} \Gamma^*\omega.
\end{equation}
For further details on Lagrange Flux ($\Flux_L$) we refer to \cite[Section 6]{Solomon13}.
\begin{cor}
\label{cor2}
Suppose $r_L$ is surjective. For every $c\in H^1(L;\R)$ there exists an $h\in H_1(L;\R)$ such that $i_L(h)\in H_1(M;\R)$ is realized as the rotation vector of a $\mathcal{M}(\phi_H)$-measure. Moreover,  
\[
\langle c,h\rangle \geq \sup_{\psi}\left( \min_{\psi_1(L)}(H)-\max_{L}(H) \right)
\]
where the supremum runs over all Lagrange isotopies $\psi:[0,1]\times L\to M$ starting at $L$ and $\Flux_L(\psi)=c$.
\end{cor}

\subsection{The non-compact setting}
\label{noncompact}
Here we will discuss the analogue of the above results in the setting of a non-compact $(M,\omega)$. These can be stated for very general $(M,\omega)$, but the best results are available for Liouville manifolds. This setting in particular generalizes the case of cotangent bundles covered by Vichery \cite{Vichery14}.

Following \cite{EntovPolterovich17} we will use the terminology that an \emph{exact} symplectic manifold $(M,\omega)$ is \emph{Liouville} if the following two conditions are met:
\begin{enumerate}[(a)]
\item
$M$ is equipped with a distinguished 1-form $\lambda$ such that $d\lambda=\omega$ and the Liouville vector field $Z\in \mathfrak{X}(M)$ defined by $i_Z \omega=\lambda$ is complete.
\item 
\label{Mcondb}
There exists a closed connected hypersurface $\Sigma \subset M$, which is transverse to $Z$ and bounds a domain $U\subset M$ whose closure is compact, such that 
\[
M=U \sqcup \left( \bigcup_{t\geq 0}\theta_t(\Sigma) \right),
\]
where $(\theta_t)_{t\in \R}$ denotes the Liouville flow generated by $Z$.
\end{enumerate}
The last condition implies that $(\Sigma,\lambda_0:=\lambda|_{\Sigma})$ is a contact manifold, called the \emph{ideal contact boundary  of $M$}. For a Liouville $(M,\omega=d\lambda)$, there is a diffeomorphism $\cup_{t\in \R}\theta_t(\Sigma)\cong (0,\infty)\times \Sigma$ which identifies $\lambda$ with $s\lambda_0$ on $(0,\infty)\times \Sigma$ (here $s$ denotes the coordinate on $(0,\infty)$). In the following we will tacitly use the identification $\cup_{t\in \R}\theta_t(\Sigma)\cong (0,\infty)\times \Sigma$ and think of $(0,\infty)\times \Sigma$ as a subset of $M$ with the property that $\lambda|_{(0,\infty)\times \Sigma}=s\lambda_0$.

Let now $(M,\omega=d\lambda)$ be a Liouville manifold and denote by $L\subset (M,\omega)$ a closed $\lambda$-exact Lagrangian submanifold.\footnote{Recall that a Lagrangian $L\subset (M,d\lambda)$ is said to be \emph{$\lambda$-exact} if $\lambda|_L=df$ for some $f\in C^{\infty}(L)$.} We will denote by $\mathcal{H}=\mathcal{H}(M,\omega)$ the class of autonomous Hamiltonians $H\in C^{\infty}(M)$ whose Hamiltonian flow $\phi_H=\{\phi_H^t\}_{t\in \R}$ is complete and satisfies the condition that $\mathcal{O}_+(X)\subset M$ is compact for every compact $X\subset M$. Here 
\[
\mathcal{O}_+(X):=\overline{\bigcup_{t\geq 0}\phi_H^t(X)}.
\]
Note that $\mathcal{H}$ contains all proper Hamiltonians. 
Given $H\in \mathcal{H}$, every measure $\mu \in \mathcal{M}$ has a well-defined action $\mathcal{A}_{H,\lambda}(\mu)\in \R$ given by\footnote{Recall that all $\mathcal{M}$-measures are required to have \emph{compact} support.}
\begin{equation}
\label{eq8}
\mathcal{A}_{H,\lambda}(\mu):=\int H- \langle \lambda, X_H \rangle \ d\mu.
\end{equation}
In order to define  
\[
\sigma_{H:L}:H^1(M;\R) \to \R
\]
for $H\in \mathcal{H}$ in this setting we make use of 
\begin{prop}
\label{Mprop1}
If $(M,\omega=d\lambda)$ is Liouville, then every class $c\in H^1_{dR}(M;\R)$ admits a representative $\eta$ such that 
\begin{equation}
\label{Meq1}
\eta|_{\{s\}\times \Sigma}=\eta|_{\{1\}\times \Sigma}\quad \forall \ s\in [1,\infty).
\end{equation}
Here the equality means the two 1-forms on $\Sigma$ are the same. In particular, the vector field $X\in \mathfrak{X}(M)$ characterized by $i_X\omega=\eta$ is complete. 
\end{prop}
The proof is more or less immediate, but is carried out in Section \ref{seqprem}. Given a class $c\in H^1_{dR}(M;\R)=H^1(M;\R)$ we define $\sigma_{H:L}(c)$ by the formula (\ref{eq2}), where $(\psi_t)_{t\in \R}$ is the flow generated by a vector field $X$ satisfying $i_X\omega=\eta$ for a closed 1-form $\eta \in c$ meeting condition (\ref{Meq1}). In Section \ref{seqprem} below we argue that the value of $\sigma_{H:L}(c)$ depends neither on the choice of $\eta$ in the class $c$ satisfying (\ref{Meq1}), nor on the choice of ideal contact boundary $\Sigma$ meeting condition (b) on page \pageref{Mcondb}. The non-compact version of Theorem 3 now reads as follows:
\begin{thm}
\label{thm2}
Let $H\in \mathcal{H}(M,d\lambda)$. Then, for every $c\in H^1(M;\R)$ and every Clarke subdifferential $h\in \partial \sigma_{H:L}(c)\subset H_1(M;\R)$, there exists a measure $\mu \in \mathcal{M}(\phi_H)$ whose action and rotation vector satisfy
\[
\mathcal{A}_{H,\lambda}(\mu)=\sigma_{H:L}(c)-\langle c,\rho(\mu) \rangle  \quad \text{and} \quad \rho(\mu)=h.
\]
\end{thm}
The following is the non-compact version of Corollary \ref{lem1}.
\begin{cor}
\label{cor3}
Let $H\in \mathcal{H}(M,d\lambda)$ and suppose $r_{L}$ is surjective. Then $\sigma_{H:L}$ descends to a locally Lipschitz function
\[
\alpha_{H:L}:H^1(L;\R)\to \R.
\]
In particular, for every $c\in H^1(L;\R)$ and every Clarke subdifferential $h\in \partial \alpha_{H:L}(c)$, there exists a measure $\mu \in \mathcal{M}(\phi_H)$ satisfying 
\begin{equation}
\label{Mateq902}
\mathcal{A}_{H,\lambda}(\mu)=\alpha_{H:L}(c)-\langle c,\rho(\mu) \rangle  \quad \text{and} \quad \rho(\mu)=i_L(h)\in H_1(M;\R).
\end{equation}
\end{cor}

\begin{rem}
	\label{Matrem901}
	Suppose now that $M=T^*L$ with $L$ a closed manifold viewed as the zero-section $L\subset T^*L$ and $\omega=d\lambda$ with $\lambda=pdq$ being the canonical Liouville one-form. A Tonelli Lagrangian $l\in C^{\infty}(TN)$ (see \cite[Definition 1.1.1]{Sorrentino15}) sets up a Legendre transform 
	\[
	\mathcal{L}:TN\to T^*N,
	\]
	so that the associated Tonelli Hamiltonian $H$ and $l$ satisfy the equation
	\[
	l(\mathcal{L}^{-1}(q,p))=\langle \lambda, X_H \rangle (q,p) -H(q,p) \quad \forall \ p\in T_q^*N.
	\]
	In particular, in this setting one sees that the condition on the action in (\ref{Mateq902}) can be written
	\begin{equation}
	\label{Mateq903}
	-\alpha_{H:L}(c)=\int l\circ \mathcal{L}^{-1}\ d\mu-\langle c,\rho(\mu) \rangle.
	\end{equation}
	As mentioned in Remark \ref{rem101}, it is known that $\alpha_{H:L}$ coincides with Mather's $\alpha$-function in this setting so (by \cite[Section 3.1]{Sorrentino15}) (\ref{Mateq903}) implies that the measures detected in Corollary \ref{cor3} are Mather measures in this setting.
\end{rem}

\begin{rem}[On the definition of $\mathfrak{M}_{H:L}$]
	\label{Matrem902}
	In the Tonelli setting from the above Remark \ref{Matrem901}, a Mather measure associated to $0\in H_1(L;\R)$ is a measure $\mu \in \mathcal{M}(\phi_H)$ which satisfies (\ref{Mateq902}) with $c=0$. So why don't we define $\mathfrak{M}_{H:L}$ as the set of measures satisfying a condition similar to (\ref{Mateq902}) with $c=0$? The reason is quite simple: In the Tonelli setting all ergodic components of a Mather measure associated to $0\in H_1(M;\R)$ are themselves Mather measures associated to $0\in H_1(M;\R)$. In the general setting which we consider here this need not be true. E.g. we have the following pathological example: Fix $p_0\in \R^n \backslash \{0\}$ and choose $r>0$ such that $0\notin \overline{B_{2r}(p_0)}$. Pick any function $h\in C^{\infty}(B_r(p_0))$ with $h(p_0)=0$ and $dh(p_0)\cdot p_0=0$, but $dh(p_0)\neq 0$. Extend $h$ to $B_r(p_0)\cup (-B_r(p_0))$ by $h(p)=h(-p)$ for all $p\in -B_r(p_0)$. Now extend $h$ to a function on all of $\R^n$ by cutting it off outside $\overline{B_{2r}(p_0)} \cup (-\overline{B_{2r}(p_0)})$. The Hamiltonian $H\in C^{\infty}(\T^n \times \R^n)$ given by $H(q,p)=h(p)$ is integrable with $\alpha_{H:\T^n}(0)=0$ and $\partial \alpha_{H:L}(0)=\{0\}$. Moreover, we have
	\[
	\alpha_{H:\T^n}(p_0)=\alpha_{H:\T^n}(-p_0)=0 \quad \& \quad 	\partial \alpha_{H:\T^n}(p_0)= -\partial \alpha_{H:\T^n}(-p_0)=\{dh(p_0)\}.
	\]
	In particular Corollary \ref{cor3} guarantees the existence of $\mathcal{M}(\phi_H)$-measures $\mu_1$ and $\mu_2$ such that $\rho(\mu_1)=dh(p_0)=-\rho(\mu_2)$ and 
	\begin{align*}
	\mathcal{A}_{H,\lambda}(\mu_1)=\alpha_{H:\T^n}(p_0)-\langle p_0,\rho(\mu) \rangle =0 \\
	\mathcal{A}_{H,\lambda}(\mu_2)=\alpha_{H:\T^n}(-p_0)+\langle p_0,\rho(\mu_2) \rangle  =0.
	\end{align*}
	Moreover, from Remark \ref{rem101} it is not hard to deduce that $\supp(\mu_1)\subset \T^n \times \{p_0\}$ and $\supp(\mu_2)\subset \T^n \times \{-p_0\}$. In particular the measure $\mu:=\frac{\mu_1+\mu_2}{2}$ is $\phi_H$-invariant and satisfies 
	\[
	\mathcal{A}_{H,\lambda}(\mu)=\frac{\mathcal{A}_{H,\lambda}(\mu_1)+\mathcal{A}_{H,\lambda}(\mu_2)}{2}=0=\alpha_{H:\T^n}(0)-\langle 0,\rho(\mu) \rangle
	\]
	and $\rho(\mu)=0\in \partial \alpha_{H:\T^n}(0)$, so $\mu$ formally satisfies (\ref{Mateq902}) with $c=0$. However, $\supp(\mu)$ clearly carries no information about the "dynamics close to the zero-section $\T^n\times \{0\}$".
	
	This example shows that in the non-Tonelli setting the requirement (\ref{Mateq902}) is not enough to guarantee that the support of $\mu$ carries information about the dynamics of $\phi_H$ close to $L$. The definition of $\mathfrak{M}_{H:L}$ in Remark \ref{rem101} guarantees that $\mathfrak{M}_{H:L}$-measures do carry such information. We do not know if $\mathfrak{M}_{H:L}$ contains all Mather measures when $H$ is Tonelli. 
\end{rem}

\begin{rem}[Superlinearity and coercivity of $\alpha_{H:L}$]
	\label{Rem1rev2}
	In Aubry-Mather theory, when $H\in C^{\infty}(T^*L)$ is Tonelli, the associated function $\alpha_{H:L}:H^1(L;\R)\to \R$ is also convex and superlinear. In particular, the set of all subdifferentials of $\alpha_{H:L}$ is all of $H_1(L;\R)$. Hence, all vectors in $H_1(L;\R)$ are realized as rotation vectors of some $\mathcal{M}(\phi_H)$-measure. In the setting of Corollary \ref{cor3}, if $H\in C^{\infty}(M)$ is a Hamiltonian for which 
	\begin{equation}
	\label{rev2eq2}
	\frac{\alpha_{H:L}(c)}{|c|} \stackrel{|c|\to \infty}{\longrightarrow} \infty
	\end{equation}
	(with respect to some norm on $H^1(L;\R)$), then the same conclusion holds true.\footnote{Use e.g. Lebourg's mean value theorem from Section \ref{secmeas} below.} It is interesting to understand which conditions should be imposed on $H$ to achieve (\ref{rev2eq2}). Since $\rho:\mathcal{M}(\phi_H)\to H_1(M;\R)$ is continuous (with respect to the weak$^*$-topology on $\mathcal{M}(\phi_H)$) it is easy to see that its image is compact if $M$ is compact (use weak$^*$-compactness of $\mathcal{M}(\phi_H)$). Hence, to achieve (\ref{rev2eq2}) one will need that $M$ be non-compact. However, even on non-compact $(M,\omega)$ one might not be able to achieve (\ref{rev2eq2}) for any $H$ as the following example shows: Consider the two symplectic manifolds $(T^*N,d\lambda)$ and $(\T^2=S^1\times S^1,\omega_0)$ for $N$ a closed manifold and $\omega_0$ the standard area form on $\T^2$ with $\int_{\T^2}\omega_0=1$. Now take $(M,\omega):=(T^*N\times \T^2, d\lambda \oplus \omega_0)$ and consider the Lagrangian $L:=N\times (S^1 \times \{0\}) \subset M$. Denote by $\eta \in \Omega^1(M)$ the closed 1-form obtained by pulling back the angle 1-form $d\varphi \in \Omega(S^1)$ (with $\int_{S^1}d\varphi=1$) by the projection $M\to \T^2\to S^1\times \{0\}$ and denote by $\psi$ the isotopy generated by the vector field $X\in \mathfrak{X}(M)$ satisfying $\iota_X \omega=\eta$. Then the flow $\psi$ is $1$-periodic, so $L=\psi_{\Z}(L)$ and c) in Theorem \ref{prop1} implies $\alpha_{H:L}(\Z [\eta])=\alpha_{H:L}(0)$ for all admissible $H\in C^{\infty}(M)$. In particular $\alpha_{H:L}(t[\eta])$ will \emph{not} tend to $\infty$ for $t\to \infty$. Of course one could consider this a non-example: In fact $\alpha_{H:L}$ naturally descends to a function on the "cylinder" $H^1(L;\R)/\Z [\eta]\cong H^1(N;\R)\times S^1$ and, viewed as such a function, (\ref{rev2eq2}) will certainly be true for some Hamiltonians $H\in C^{\infty}(M)$.
	
	On the other hand, if $H\in C^{\infty}(T^*N)$ satisfies 
	\[
	\frac{H(q,p)}{|p|} \stackrel{|p|\to \infty}{\longrightarrow}\infty
	\]
	and $H^1(N;\R)$ admits a basis represented by closed 1-forms, all of which are nowhere vanishing, then it is easy to check that (\ref{rev2eq2}) follows. This is in particular the case on $T^*\T^n$. See also \cite{Viterbo10} for an in-depth discussion of this case. 
\end{rem}
\section{Comparison with Mather's theory}
\label{Mathersec}
In this section we compare properties of the support of the measures whose existence is guaranteed by Theorem \ref{thm1} and \ref{thm2} to those which arise in Mather's theory and place our results in a historical context. To accomplish this, it will suffice to consider $(M,\omega)=(T^*\T^2,d\lambda)$.
We will construct an $H$ on $T^*\T^2$ for which the support of $\mathfrak{M}_{H:L}$-measures becomes extremely "wild". This phenomenon is closely related to diffusion phenomena such as Arnold' diffusion and superconductivity channels \cite{BounemouraKaloshin14}. Studying $\mathfrak{M}_{H:L}$-measures in this setting was generously suggested to me by Vadim Kaloshin.

Consider $\R^2(p)$ as well as $\T^2=\R^2 /\Z^2$ equipped with the (mod $\Z^2$) coordinate $q$. Given sufficiently smooth functions $H_0:\R^2 \to \R$ and $F:\T^2 \times \R^2 \to \R$ as well as $\epsilon \geq 0$ we consider the Hamiltonian
\[
H_{\epsilon}(q,p):=H_0(p)+\epsilon F(q,p)
\]
on $(T^*\T^2=\T^2 \times \R^2, dp\wedge dq)$. Throughout this section we make use of the canonical identification $H^1(T^*\T^2,\R)\cong \R^2$. The flow $\phi_{H_0}$ is integrable in the sense that it leaves every Lagrangian torus of the form $T(p):=\T^2 \times \{ p \}$ invariant. KAM theory guarantees that, if $H_0$ is non-degenerate in the sense that its Hessian satisfies
\[
\det(\Hess(H_0))(p)\neq 0 \quad \forall \ p\in \R^2,
\]
then, for all small enough $\epsilon>0$, many Lagrangian tori $T(p)$ persist. More precisely, if the rotation vector of $\phi_{H_{0}}|_{T(p)}$ is Diophantine and $\epsilon >0$ is small enough, then there is a Lagrangian torus which is invariant for $\phi_{H_{\epsilon}}$ close to $T(p)$. We are interested in what happens to the invariant torus $T(p)$ for $\epsilon>0$ when $\phi_{H_0}|_{T(p)}$ does \emph{not} have Diophantine rotation vector. Does it (at least partially) persist or does it disappear? Mather \cite{Mather91} studied the case  
\begin{equation}
\label{equa11}
\det(\Hess(H_0))(p)>0 \quad \forall \ p\in \R^2.
\end{equation}
In this case all $T(p)$ \emph{partially} persist. The invariant set studied by Mather, which is to be thought of as an "avatar" of $T(p)$, is the $p$-\emph{Mather set} $\supp(\mathfrak{M}_{H_{\epsilon}:T(p)})\subset \T^2 \times \R^2$.\footnote{It is a consequence of Mather's theory (or Theorem \ref{thm2}) that this set is always $\neq \emptyset$.} This interpretation of the Mather set relies on two key properties exhibited by $\supp(\mathfrak{M}_{H_{\epsilon}:T(p)})$ under assumption (\ref{equa11}):
\begin{itemize}
	\item 
	Graph property: $\supp(\mathfrak{M}_{H_{\epsilon}:T(p)})\subset \T^2 \times \R^2$ is the graph of a Lipschitz function defined on a subset of $T(p)$.
	\item 
	Localization property: In a suitable sense $\supp(\mathfrak{M}_{H_{\epsilon}:T(p)})$ "is close" to (a subset of) $T(p)$ when $\epsilon >0$ is small.
\end{itemize}
The precise statement of the localization property takes different forms, depending on the dynamics of $\phi_{H_0}|_{T(p)}$. Regardless of the dynamics of $\phi_{H_0}|_{T(p)}$ it is not hard to see that, under assumption \ref{equa11}, one will always have
\begin{equation}
\label{Rev2eq1}
\supp(\mathfrak{M}_{H_{0}:T(p)})\cap \liminf_{\epsilon \downarrow 0}\supp(\mathfrak{M}_{H_{\epsilon}:T(p)})\neq \emptyset,
\end{equation}
where $\liminf$ is to be understood in the sense of Kuratowski. This follows simply because, under assumption (\ref{equa11}), $\mathfrak{M}_{H_{\epsilon}:T(p)}$-measures minimize a Lagrange functional. Mather \cite[Section 5]{Mather91} studied the case when the rotation vector of $\phi_{H_0}|_{T(p)}$ meets a Diophantine condition (i.e. when $T(p)$ is a KAM torus). He found that, under this assumption, $\supp(\mathfrak{M}_{H_{\epsilon}:T(p)})$ converges to $T(p)$ in the Hausdorff distance as $\epsilon \to 0$. The Diophantine condition was later relaxed by Bernard \cite{Bernard00,Bernard10}. He studied the case when the dynamics $\phi_{H_0}|_{T(p)}$ is resonant, but $T(p)$ is foliated by invariant subtori on which the dynamics of $\phi_{H_0}$ are "sufficiently irrational". In this case he found that $\supp(\mathfrak{M}_{H_{\epsilon}:T(p)})$ Hausdorff converges to one of the invariant subtori of $T(p)$ as $\epsilon \to 0$ \cite[Theorem 1]{Bernard00}.

Summing up, Mather's theory (morally speaking) guarantees that, if (\ref{equa11}) is satisfied, then the system cannot be \emph{too} unstable in the sense that every $T(p)$ will partially survive small perturbations of $H_0$. In a different direction, Arnold' conjectured \cite{Arnold64}, \cite{Arnold94} that the Hamiltonian flow generated by a generic Hamiltonian $H\in C^{\infty}(T^*\T^n)$, with $n\geq 3$, will exhibit diffusing orbits (this phenomenon is today known as \emph{Arnold' diffusion}). Hence, if Arnold's conjecture holds true, then all systems with more than two degrees of freedom will exhibit \emph{some} instability.

When 
\begin{equation}
\label{equa2}
\det(\Hess(H_0))(p)<0 \quad \forall \ p\in \R^2
\end{equation}
the candidate for a perturbation of $T(p)$ is again $\supp(\mathfrak{M}_{H_{\epsilon}:T(p)})$. Do these sets satisfy the graph property or the localization property? To our knowledge Herman \cite{Herman88} was the first to observe that, in general, in the case (\ref{equa2}) $\supp(\mathfrak{M}_{H_{\epsilon}:T(p)})$ will \emph{not} be a Lipschitz graph over $T(p)$ (see also \cite{Chen92} for an English exposition of Herman's example). Below we construct an example showing that not only can the graph property be violated, but so can the localization property. In fact the example violates even the simplest localization property (\ref{Rev2eq1}) in the worst possible way. The dynamics responsible for the failure of this kind of convergence is a very fast type of diffusion arising from so-called \emph{superconductivity channels}. We learned about this phenomenon from Bounemoura-Kaloshin's \cite{BounemouraKaloshin14}. A first indication that symplectic methods can be used to study this phenomenon was found in \cite{EntovPolterovich17}. The concept of superconductivity channels goes back to the work of Nekoroshev \cite{Nekoroshev77}, \cite{Nekoroshev79} who discovered that for every $x=(q,p)\in \T^2 \times \R^2$, $\phi_{H_{\epsilon}}^t(x)$ remains close to $\phi_{H_0}^t(x)\in T(p)$ for $|t|\leq e^{c\epsilon^{-a}}$ for some positive constants $c,a>0$, as long as $\epsilon>0$ is sufficiently small and $H_{0}$ is \emph{steep}.\footnote{The "closeness" in this Nekoroshev estimate is in fact quantitative. For precise details we refer to \cite[Section 1.1.2]{BounemouraKaloshin14}} By a result due to Ilyashenko \cite{Ilyashenko86}, the steepness condition can be phrased as requiring that the restriction of $H_0$ to any $1$-dimensional linear subspace of $\R^2$ has only isolated critical points. Hence, the philosophy we follow is that, in order to find diffusion which is fast enough to "push" $\supp(\mathfrak{M}_{H_{\epsilon}:T(p)})$ far away from $T(p)$, we should choose $H_0$ \emph{non}-steep.

\subsection{The example} 
\label{secex}
Consider in symplectic coordinates $\theta=(\theta_1,\theta_2)\in \T^2$ and $I=(I_1,I_2)\in \R^2$ the Hamiltonian
\[
H_{\epsilon}(\theta,I) =I_1I_2 +\epsilon \varphi(I_1)\sin(2\pi \theta_1),
\]
where $\varphi:\R \to [0,1]$ is a smooth function with $\varphi(t)=\varphi(-t)$ and
\[
\varphi(s)=
\begin{cases}
1 & \text{if} \ |s|\leq K \\
0 & \text{if} \ |s|>K+1
\end{cases}, \quad 
\varphi'(s)
\begin{cases}
\geq 0 & \text{if} \ s\leq 0 \\
\leq 0 & \text{if} \ s\geq 0
\end{cases}
\]
for some large constant $K>0$ (see Figure \ref{Mfig1}). Note that, if we view $H_0$ as a function on $\R^2$, then $\det(\Hess(H_0))(I)<0$ for all $I\in \R^2$, so $H_{\epsilon}$ fits into the framework we discussed above.
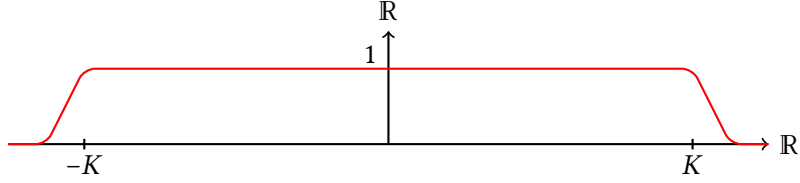
\begin{figure}[h]
	\begin{tikzpicture}[scale=1]
\draw [thick, ->] (-5,0) -- (5,0);
\draw [thick, ->] (0,0) -- (0,1.5);
\draw [thick] (-4,-0.07) -- (-4,0.07);
\draw [thick] (4,-0.07) -- (4,0.07);
\node [right] at (5,0) {\small $\R$};
\node [above] at (0,1.5) {\small $\R$};
\node [left] at (0,1.2) {\small $1$};
\node [below] at (-4,0) {\small $-K$};
\node [below] at (4,0) {\small $K$};
\draw [thick, red, rounded corners] (-5,0) -- (-4.5,0) -- (-4.0,1) -- (0,1) -- (4,1) -- (4.5,0) -- (5,0);
\end{tikzpicture}
	\caption{Graph of $\varphi$ indicated in red.}
	\label{Mfig1}
\end{figure}
The resulting equations of motion are given by 
\begin{equation}
\label{equa4}
\begin{cases}
\dot{I}_1=-2\pi \epsilon \varphi (I_1)\cos(2\pi \theta_1)\\
\dot{I}_2=0 \\
\dot{\theta}_1=I_2+\epsilon \varphi'(I_1)\sin(2\pi \theta_1) \\
\dot{\theta}_2=I_1. 
\end{cases}
\end{equation}
In particular, we see that $\mathcal{O}_+(X)$ is compact for every compact $X$, so Corollary \ref{cor3} applies to this system.
\begin{rem}
\label{rem000}
The superconductivity channels $\{I_1=0\}\cup \{I_2=0\}$ give rise to diffusion: Consider initial conditions 
\[
(\theta_1(0),\theta_2(0),I_1(0),I_2(0))
\]
such that $I_2(0)=0$, $|I_1(0)|<K$ and $\sin(2\pi \theta_1(0))=0$. Then $t\mapsto \theta_1(t)$ is constant and 
\[
\dot{I}_1=\pm 2\pi \epsilon \varphi(I_1),
\]
so the function $t\mapsto |I_1(t)-I_1(0)|$ is strictly increasing (assuming $\epsilon>0$). We denote by 
\[
\mathcal{D}:=\{ (\theta_1,\theta_2,I_1,I_2)\ |\ \sin(2\pi \theta_1)=0, \ I_2=0, \ |I_1|<K \}
\]
the set of initial conditions with $|I_1(0)|<K$ which diffuse. Poincar{\'e} recurrence implies
\begin{equation}
\label{equa9}
\mu(\mathcal{D})=0 \quad \forall \ \mu \in \mathcal{M}(\phi_{H_{\epsilon}}),
\end{equation}
for every $\epsilon>0$.
\end{rem}
Note that if $\epsilon=0$, then (by Remark \ref{rem101} and the fact that each $T(I_1,I_2)$ is $\phi_{H_0}$-invariant) 
\[
\supp(\mathfrak{M}_{H_0:T(I_1,I_2)})\subset T(I_1,I_2).
\]
In contrast, when $\epsilon >0$ we have 
\begin{prop}
\label{propdiff}
Fix $I\in (-K,K)$ and denote by $B_K\subset \R^2$ the open $K$-ball centered at $0$. For every $\epsilon >0$ we have 
\[
\supp(\mathfrak{M}_{H_{\epsilon}:T(I,0)})\cap (\T^2 \times B_K)=\emptyset.
\]
\end{prop}
\begin{figure}[h]
	\centering
	\begin{tikzpicture}[scale=1]
\draw [thick] (0,0) circle [radius=2];
\node [right] at (1,2) {\small $B_K$};
\draw [thick, ->] (-2.5,0) -- (2.5,0);
\draw [thick, ->] (0,-2.5) -- (0,2.5);
\node [below] at (2.5,0) {\small $I_1$};
\node [right] at (0,2.5) {\small $I_2$};
\draw [fill, red] (0,0) circle [radius=0.08];
\draw [fill, blue] (2.05,-0.05) rectangle (2.3,0.05);
\draw [fill, blue] (-2.05,-0.05) rectangle (-2.3,0.05);
\end{tikzpicture}
	\caption{The projection of $\supp(\mathfrak{M}_{H_0:T(0,0)})$ to $\R^2$ is contained in the red spot in the center. For all $\epsilon>0$, the projection of $\supp(\mathfrak{M}_{H_{\epsilon}:T(0,0)})$ to $\R^2$ is contained in the blue regions.}
	\label{fig03}
\end{figure}
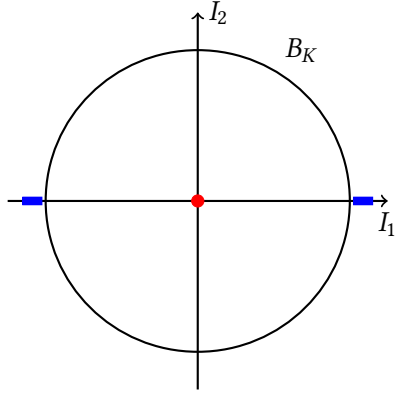
\begin{rem}
	This proposition exhibits the violation of the localization property for $\mathfrak{M}_{H_{\epsilon}:L}$ when $H_0$ fails to meet the convexity condition (\ref{equa11}) in the following sense: If we impose the convexity hypothesis on $H_0$, a result due to Bernard \cite{Bernard00} implies that $\mathfrak{M}_{H_{\epsilon}:T(I,0)}$ Hausdorff-converges to the circle 
	\[
	\{(\tfrac{1}{4},\theta_2,I,0) \ |\ \theta_2\in S^1\}\subset T(I,0)
	\]
	as $\epsilon \downarrow 0$ for all $I\in \R \backslash \mathbb{Q}$ with $|I|<K$. Clearly, Proposition \ref{propdiff} implies that this is not the case for $\mathfrak{M}_{H_{\epsilon}:T(I,0)}$ in our setting. To present how our example fits into Bernard's theory \cite{Bernard00} we will from now on impose the assumption that $H_0$ be convex\footnote{Which, of course, it clearly is \emph{not}.} and check that, after adding this hypothesis to our example, it formally fits into Bernard's framework. In Bernard's terminology our $\phi_{H_{0}}^1$ is denoted by $\Phi_0$, and its lift to $\R^2\times \R^2$ is denoted by $\phi_0$. Bernard's \textsc{Hypothesis 1} requires that there is a diffeomorphism $w:\R^2\to \R^2$ such that\footnote{We denote by $\Theta=(\Theta_1,\Theta_2)$ standard coordinates on $\R^2$ which project to $(\theta_1,\theta_2)$ on $\T^2$.} 
	\[
	\phi_0(\Theta,I)=(\Theta+w(I),I).
	\]
	Clearly, in our case $w(I_1,I_2)=(-I_2,-I_1)$. A generating function $S_0:\R^2 \times \R^2 \to \R$ for $\phi_0$ is a smooth function such that $\phi_0(\Theta(0),I(0))=(\Theta(1),I(1))$ if and only if 
	\[
	I(0)=\partial_{1}S_0(\Theta(0),\Theta(1)) \quad \text{and} \quad I(1)=-\partial_{2}S_0(\Theta(0),\Theta(1)).
	\]
	In our example we can take 
	\[
	S_0(\Theta(0),\Theta(1))=H_0(\Theta(1)-\Theta(0))=(\Theta_1(1)-\Theta_1(0))(\Theta_2(1)-\Theta_2(0)),
	\] 
	so our imposed convexity condition on $H_0$ amounts to Bernard's \textsc{Hypothesis 2}. 
	
	Bernard studies the Aubry-Mather set for the diffeomorphism $\Phi_{\epsilon}:T^*\T^2 \to T^*\T^2$ whose pull-back $\phi_{\epsilon}$ to $\R^2\times \R^2$ has as a generating function 
	\[
	S_{\epsilon}(\Theta(0),\Theta(1))=S_0(\Theta(0),\Theta(1))+\epsilon P(\Theta(0),\Theta(1)),
	\] 
	for a function $P:\R^2 \times \R^2\to \R$ which satisfies an exactness condition (Bernard's \textsc{Hypothesis 3}). Moreover, as he explains, one can assume that $S_{\epsilon}$ is "standard at $\infty$" (\textsc{Hypothesis 4}). Under our imposed convexity assumption on $H_0$, $\phi_{H_{\epsilon}}^1$ would also admit a generating function $S_{\epsilon}$ of the type above, and the two latter assumptions would also hold true in our setting (after perhaps adapting $S_{\epsilon}$ "at $\infty$"). Fix now $I\in \R \backslash \mathbb{Q}$ such that $|I|<K$. Then the torus $T(I,0)$ is foliated by an $S^1$-family of $\phi_{H_0}^1$-invariant circles $S^1(\theta_1)=\{(\theta_1,\theta_2,I,0) \ | \ \theta_2 \in S^1\},\ \theta_1\in S^1$. Bernard introduces the \emph{averaged perturbation} which in our situation is given by 
	\[
	\mathcal{P}(\Theta_1)=\int_{S^1}P(\Theta_1,\Theta_2,\Theta_1,\Theta_2-I) d\Theta_2.
	\]
	$\mathcal{P}$ descends to a function $\mathcal{P}:S^1 \to \R$ and Bernard's \textsc{Hypothesis 5} requires that this function have a unique non-degenerate minimum $\theta_1^*\in S^1$. Once we have checked that this is the case in our example (under our imposed convexity assumption) Bernard's \cite[Theorem 1]{Bernard00} would imply that $\mathfrak{M}_{H_{\epsilon}:T(I,0)}$ Hausdorff-converges to the subset $S^1(\theta_1^*)\subset T(I,0)$ as $\epsilon \downarrow 0$. 
	
	To see that $\mathcal{P}$ would have to have a unique non-degenerate minimum on $S^1$ in our example, note that $S_0$ vanishes on $T(I,0)$, so in fact we have $S_{\epsilon}=\epsilon P$ on $T(I,0)$. Recall \cite[Section 9.1]{McDuffSalamon17} that the generating function $S_{\epsilon}$ of the lift $\phi_{\epsilon}$ of $\phi_{H_{\epsilon}}^1$ is (up to addition by a constant) given by the symplectic action in the following sense: Given $(\Theta(0),\Theta(1))\in \R^2 \times \R^2$ there exist (due to the convexity assumption) unique $(I(0),I(1))\in \R^2 \times \R^2$ such that $\phi_{\epsilon}(\Theta(0),I(0))=(\Theta(1),I(1))$ and 
	\begin{align*}
	S_{\epsilon}(\Theta(0),\Theta(1))&=\int_0^1 \langle \lambda ,X_{H_{\epsilon}}\rangle (\Theta(t),I(t))-H_{\epsilon}(\Theta(t),I(t))\ dt\\
	&=-\epsilon \int_0^1 (\varphi(I_1(t))-I_1(t)\varphi'(I_1(t)))\sin(2\pi \Theta_1(t)) \ dt
	\end{align*}
	Now from (\ref{equa4}) it follows that, if $\epsilon >0$ is small enough then $|I_1(t)|<K$ for all $t\in [0,1]$, so that in particular $\Theta_1(t)=\Theta_1(0)$ for all $t\in [0,1]$ (since $I_2(t)=0$). Hence, for such $\epsilon>0$ we have 
	\begin{align*}
	S_{\epsilon}(\Theta(0),\Theta(1))=-\epsilon \sin(2\pi \Theta_1(0)),
	\end{align*}
	and thus 
	\[
	\mathcal{P}(\theta_1)=-\sin(2\pi \theta_1), \quad \theta_1\in S^1.
	\]
	So clearly $\mathcal{P}$ would have a unique non-degenerate minimum at $\theta_1^*=\frac{1}{4}$. This finishes the demonstration that our example, formally speaking, fits into Bernard's framework if only one imposes a convexity condition on $H_0$. 
\end{rem}

The rest of this section will be spent proving Proposition \ref{propdiff}. To make the notation a little less heavy we will write $\alpha_{H_{\epsilon}}(I_1,I_2)=\alpha_{H_{\epsilon}:T(0,0)}(I_1,I_2)$ for $(I_1,I_2)\in \R^2$.
\begin{rem}
The conclusion of Proposition \ref{propdiff} resonates well with classical insight from symplectic topology. For example, it is nowadays well understood that $\dim_{\R}H_*(M;\R)\leq \# \text{Fix}(\phi)$ for non-degenerate $\phi \in \Ham(M,\omega)$ (this is one of the celebrated Arnol'd conjectures). Symplectic topology can verify this inequality, but it cannot say anything about where in $M$ the fixed points are located.
\end{rem}
\begin{lemma}
\label{lemmat1}
For every $\epsilon \geq 0$ we have
\[
\alpha_{H_{\epsilon}}(I,0)=0 \quad \forall \ I\in \R.
\]
\end{lemma}
The proof of this lemma (which uses nothing but standard properties of Lagrangian spectral invariants) is presented in Section \ref{secproof2}. The proof of Proposition \ref{propdiff} consists in studying carefully the solutions to (\ref{equa4}). From Remark \ref{rem101} and the proof of Theorem \ref{thm1} we know that, in order to study the support of $\mathfrak{M}_{H_{\epsilon}:T(I,0)}$-measures for $\epsilon>0$, it suffices to study the support of $\mu \in \mathcal{M}(\phi_{H_{\epsilon}})$ which arise as the weak$^*$-limits of sequence $(\mu_{\varsigma})_{\varsigma\in \N}\subset \mathcal{M}$ characterized by (\ref{equa12}). In the current setting we can write $(x_{\varsigma})_{\varsigma \in \N}=(\theta^{\varsigma}_1(0),\theta_2^{\varsigma}(0),I_1^{\varsigma}(0),I_2^{\varsigma})_{\varsigma \in \N}$ ($I_2$ is an integral of motion). The conditions described in Remark \ref{rem101} and in the proof of Theorem \ref{thm1} now amount to 
\begin{align}
&(I_1^{\varsigma}(0),I_2^{\varsigma})\stackrel{\varsigma \to \infty}{\longrightarrow} (I,0) \label{equa5}\\
&\frac{\mathcal{A}_{H_{\epsilon}:T(I_1^{\varsigma}(0),I_2^{\varsigma}(0))}(\gamma_{\varsigma})}{k_{\varsigma}}\stackrel{\varsigma \to \infty}{\longrightarrow}\alpha_{H_{\epsilon}}(I,0)=0. \label{equa7}
\end{align}
From the computation in Remark \ref{Mrem1} on page \pageref{Mrem1} below, it follows that
\[
\frac{\mathcal{A}_{H_{\epsilon}:T(I_1^{\varsigma}(0),I_2^{\varsigma}(0))}(\gamma_{\varsigma})}{k_{\varsigma}}=\mathcal{A}_{H_{\epsilon}:\lambda}(\mu_{\varsigma})+\int \langle \eta_{\varsigma},X_{H_{\epsilon}} \rangle d\mu_{\varsigma},
\]
where the closed 1-form $\eta_{\varsigma}$ is given by 
\[
\eta_{\varsigma}=I_1^{\varsigma}(0)d\theta_1+I_2^{\varsigma}(0)d\theta_2.
\]
Hence, as a consequence of the following claim we conclude that (\ref{equa7}) amounts to saying
\begin{equation}
\label{MEQ2}
\mathcal{A}_{H_{\epsilon}:\lambda}(\mu)=\lim_{\varsigma \to \infty} \mathcal{A}_{H_{\epsilon}:\lambda}(\mu_{\varsigma})=0.
\end{equation}
\begin{claim}
	$\lim_{\varsigma \to \infty}\int \langle \eta_{\varsigma},X_H \rangle d\mu_{\varsigma}=0$
\end{claim}
\begin{proof}
	Suppose for contradiction that this were not the case. Then the rotation vector $\rho(\mu)\in \R^2$ of the weak$^*$-limit $\mu$ of $(\mu_{\varsigma})_{\varsigma \in \N}$ satisfies $ \langle Id\theta_1,\rho(\mu) \rangle \neq 0$. Since $\supp(\mu)\subset \{I_2=0\}$, this implies the existence of an ergodic measure $\nu \in \mathcal{M}(\phi_{H_{\epsilon}})$ with $\supp(\nu)\subset \{I_2=0\}$ such that $\langle Id\theta_1,\rho(\nu) \rangle \neq 0$. By Birkhoff's ergodic theorem $\nu$-almost every initial condition $x$ would then satisfy 
	\begin{equation}
	\label{MEQ1}
	\frac{1}{T}\int_0^T \langle Id\theta_1,X_{H_{\epsilon}} \rangle \phi_{H_{\epsilon}}^t(x)\ dt\stackrel{T\to \infty}{\longrightarrow}C 
	\end{equation}
	for some constant $C\neq 0$. However, as pointed out in Remark \ref{rem000}, if $I_2=0$ and $\sin(2\pi \theta_1(t_0))=0$ for some $t_0$ then $\sin(2\pi \theta_1(t))=0$ for all $t$. In particular we have that either $\sin(2\pi \theta_1(t))\leq 0$ for all $t$ or $\sin(2\pi \theta_1(t))\geq 0$ for all $t$. In either case $d\theta_1$ is an exact 1-form on a neighbourhood of $\{\phi_{H_{\epsilon}}^t(x)\ |\ t\in \R\}$, which contradicts (\ref{MEQ1}). This contradiction finishes the proof of the claim.
\end{proof}
Since 
\begin{align*}
\mathcal{A}_{H_{\epsilon}:\lambda}(\mu_{\varsigma})= \int H_{\epsilon}-\langle \lambda,X_{H_{\epsilon}}\rangle \ d\mu_{\varsigma}=\int -I_1I_2+\epsilon (\varphi(I_1)-I_1\varphi'(I_1))\sin(2\pi \theta_1) \ d\mu_{\varsigma},  
\end{align*}
(\ref{MEQ2}) can be written
\begin{align}
\label{MEQ3}
0=\mathcal{A}_{H_{\epsilon}:\lambda}(\mu)=& \epsilon \int (\varphi(I_1)-I_1\varphi'(I_1))\sin(2\pi \theta_1) \ d\mu. 
\end{align}
To further understand $\supp(\mu)$, note that energy-preservation (i.e. $H_{\epsilon}(x)=H_{\epsilon}(\phi_{H_{\epsilon}}^t(x))$) implies the estimate 
\begin{equation}
\label{Meq6}
|\varphi(I_1^{\varsigma}(t))\sin(2\pi \theta_1^{\varsigma}(t))-\sin(2\pi \theta_1^{\varsigma}(0))|\leq \frac{2K}{\epsilon}|I_2^{\varsigma}|.
\end{equation}
For this computation we assume that $|I_1^{\varsigma}(0)|<K$, which will be true for large enough $\varsigma$ by (\ref{equa5}) and the assumption that $|I|<K$. 

We will now use the estimates (\ref{MEQ3}) and (\ref{Meq6}) to prove Proposition \ref{propdiff}.\footnote{We thank the careful referee for suggesting the following elegant proof as a drastic simplifications of our original (admittedly rather clumsy) proof.} 
We first claim that these two estimates together imply that
\begin{equation}
\label{MEQ4}
\sin(2\pi \theta_1^{\varsigma}(0))\stackrel{\varsigma \to \infty}{\longrightarrow} 0.
\end{equation}
Indeed, if we could extract a subsequence with $\sin(2\pi \theta_1^{\varsigma}(0)) \to C\neq 0$ then (\ref{Meq6}) would imply that the functions $I_1\mapsto \varphi(I_1)$ and $\theta_1 \mapsto \sin(2\pi \theta_1)$ were bounded away from $0$ on $\supp(\mu)$, and that the latter had a constant sign (equal to $\text{Sign}(C)$) on $\supp(\mu)$. In particular (since $-I_1\varphi'(I_1)\geq 0 \ \forall \ I_1$) the function $I_1\mapsto (\varphi(I_1)-I_1\varphi'(I_1))$ would be positive and bounded away from $0$ on $\supp(\mu)$, so one would arrive at a contradiction to (\ref{MEQ3}). Together, (\ref{Meq6}) and (\ref{MEQ4}) now imply 
\begin{equation*}
\supp(\mu)\cap \{\varphi_1(I_1)>0\}\subset \{\sin(2\pi \theta_1)=0\},
\end{equation*}
which by Remark \ref{rem000} implies $\supp(\mu)\cap \{\varphi_1(I_1)>0\}=\emptyset$. Hence, Proposition \ref{propdiff} follows.

\section{Symplectic "Mather theory" in the absence of intersections}
\label{secC0}
From the point of view of the previous sections, it is tempting to say that the existence of Mather-like measures is a consequence of symplectic intersection phenomena. In this section we study invariant measures in the "absence of intersections", e.g. when every compact subset of $(M,\omega)$ can be displaced by a Hamiltonian diffeomorphism. More precisely, we develop a $C^0$-approach to studying invariant measures of (autonomous) Hamiltonian systems using ideas due to Entov-Polterovich \cite{EntovPolterovich17} and Polterovich \cite{Polterovich14}. Our main application is to the study of Hamiltonian systems on twisted cotangent bundles and $\R^{2n}$. The setup we consider here is slightly different from the one in the previous section. Here $(M,\omega)$ will denote a symplectic manifold which is either closed or geometrically bounded (see e.g. \cite{AudinLafontaine94}) and $L\subset (M,\omega)$ will denote any closed (i.e. compact without boundary) connected Lagrangian submanifold. Consider an autonomous Hamiltonian $H\in C^{\infty}(M)$ with complete flow $\phi_H=\{\phi_H^t\}_{t\in \R}$. Since $H$ is an integral of motion, sublevel sets $\Sigma_k=\Sigma_k(H):=\{H<k\}$ are invariant. As discussed previously, every $\mu \in \mathcal{M}(\Sigma_k;\phi_H)$ has a well-defined rotation vector or asymptotic cycle 
\[
\rho(\mu)\in H_1(\Sigma_k;\R)
\] 
defined by (\ref{eq10}). In the following we denote by $e(X)\in [0,\infty]$ the \emph{displacement energy} of a subset $X\subset M$ and by $e_S(X)$ the \emph{stable displacement energy} of $X$. We recall that 
\[
e(X):=\inf \left\{  || H || :\ H\in C^{\infty}_c([0,1]\times M) \ \text{ s.t.} \ \phi_H^1(\overline{X})\cap \overline{X}=\emptyset \right\},
\]
where we use the convention that the infimum over $\emptyset$ equals $\infty$ and 
\[
| | H | |:=\int_0^1 \max_M(H_t)-\min_M(H_t)\ dt
\]
denotes the Hofer norm of $H$. By definition, $e_S(X)=e(X\times S^1)$, where $S^1\subset T^*S^1$ denotes the 0-section and $X\times S^1$ is viewed as a subset of the symplectic manifold $(M\times T^*S^1,\omega \oplus \omega_0)$, with $\omega_0$ the canonical symplectic form on $T^*S^1$. Clearly, $e_S(X)\leq e(X)$ for all $X\subset M$ with strict inequality in several important examples \cite{Schlenk06}.

\begin{thm}
	\label{cor00}
	Suppose $(M,\omega)$ is weakly exact and let $H\in C^{\infty}(M)$ be proper and bounded from below. Fix an energy value $k\in \R$ such that $\Sigma_k$ is stably displaceable. Suppose $L\subset \Sigma_k \subset (M,\omega)$ is a Lagrangian with Abelian $\pi_1(L)$ and fix $\tilde{c}\in H^1(L;\R)$ such that $\partial \tilde{c}|_{\pi_2(M,L)}\equiv [\omega]|_{\pi_2(M,L)}$. Then for every $a\in H^1(\Sigma_k;\R)$ satisfying
	\begin{equation}
	\label{Meq5}
	a|_L\in (e_S(\Sigma_k),\infty)\cdot H^1(L;\Z)-\tilde{c} \quad (\subset H^1(L;\R))
	\end{equation}
	and every $\epsilon \in (0,k-\max_{L}H)$ there exists a measure $\mu \in \mathcal{M}(\Sigma_k;\phi_H)$ satisfying the two conditions
	\begin{align*}
	\langle a,\rho(\mu) \rangle &\leq \max_L H-k+\epsilon  \\
	H(\supp(\mu))&\subset [\max_L(H)+\tfrac{\epsilon}{2},k-\tfrac{\epsilon}{2}].
	\end{align*}
\end{thm}
In this statement, $(e_S(\Sigma_k),\infty)\cdot H^1(L;\Z)-\tilde{c}$ denotes the subset $\{t\cdot c-\tilde{c}\ |\ c\in H^1(L;\Z),\ t\in (e_S(\Sigma_k),\infty)\}$ of $H^1(L;\R)$ and $\partial :H^1(L;\R)\to H^2(M,L;\R)$ denotes the cohomological boundary map. Lemma \ref{lem00} on page \pageref{lem00} below guarantees the existence of a $\tilde{c}\in H^1(L;\R)$ such that $\partial \tilde{c}|_{\pi_2(M,L)}=[\omega]|_{\pi_2(M,L)}$ when $(M,\omega)$ is weakly exact (i.e. $\omega|_{\pi_2(M)}\equiv 0$) and $\pi_1(L)$ is Abelian.
\begin{rem}
Of course, a class $a\in H^1(\Sigma_k;\R)$ satisfying (\ref{Meq5}) might not exist. However, it does exist if the restriction map $H^1(\Sigma_k;\R)\to H^1(L;\R)$ is surjective.
\end{rem}
Theorem \ref{cor00} is in fact an application of a general existence result obtained in Theorem \ref{thm00} below. Before discussing the abstract setting, we present another consequence of Theorem \ref{thm00} (Theorem \ref{thm01}) as well as an example. To state Theorem \ref{thm01} we need

\begin{defi}
	\label{def00}
	Fix $c\in H^1(L;\R)$. The \emph{(Lagrangian) $(L,c)$-shape} of a subset $X\subset M$ is defined as the subset of $H^1(X;\R)$ consisting of those classes $a$ for which there exists a Lagrange isotopy $\psi:[0,1]\times L\to M$ starting at $L$ and ending in $X$ in the sense that $\psi_1(L)\subset X$ and which in addition satisfies the condition
	\[
	\Flux_{L}(\psi)-c=\psi_1^*a\in H^1(L;\R).
	\]
	We will denote the $(L,c)$-shape of $X$ by $\sh(X;L,c)\subset H^1(X;\R)$. By the \emph{$L$-shape of $X$} we simply mean the $(L,0)$-shape of $X$ and use the shorthand notation $\sh(X;L):=\sh(X;L,0)$. 
\end{defi}
Of course it might be that $\sh(X;L,c)=\emptyset$.
\begin{rem}
	The notion of \emph{symplectic shape} was introduced by Sikorav \cite{Sikorav89} \cite{Sikorav91} and later studied by Eliashberg \cite{Eliashberg91}. The above definition is "modelled" on this classical concept, but also makes sense for closed symplectic manifolds. In case $(T^*N,\omega=d\lambda)$ is a cotangent bundle, our definition of $(N,c)$-shape of $X\subset T^*N$ should be thought of as corresponding to the shape of $X$ in the sense of \cite{Eliashberg91} computed with respect to the primitive $\lambda-\eta_c$ of $\omega$, where $\eta_c$ is a closed $1$-form representing $c\in H^1(N;\R)=H^1(T^*N;\R)$.
\end{rem}

For the next result we will use the following notation: Given our closed Lagrangian $L\subset (M,\omega)$ and a class $c\in H^1(L;\R)$, we denote by $\gamma_L(c)$ the positive generator of the subgroup
\begin{equation}
\label{eqrev2eq1}
\mathcal{G}_L(c):=\langle [\omega]-\partial c, \pi_2(M,L)\rangle \leq \R, 
\end{equation}
if this subgroup is discrete and non-trivial. If $\mathcal{G}_L(c)$ is trivial (i.e. $\mathcal{G}_L(c)=\{0\}$) we set $\gamma_L(c)=+\infty$ and we set $\gamma_L(c)=0$ if $\mathcal{G}_L(c)$ is \emph{not} discrete. We say that $L$ is \emph{rational} if $\gamma_L(0)\in (0,\infty)$.
\begin{thm}
\label{thm01}
	Let $L\subset (M,\omega)$ be a closed Lagrangian and let $H\in C^{\infty}(M)$ be proper and bounded from below. Suppose $k\in \R$ is an energy value and $c\in H^1(L;\R)$ such that 
	\begin{equation}
	\label{ref2eq2}
	e_S(\Sigma_k)<\gamma_L(c).
	\end{equation}
	Then for every $a\in \sh(\Sigma_k;L,c)\subset H^1(\Sigma_k;\R)$ there exists an ergodic $\mu \in \mathcal{M}(\Sigma_k;\phi_H)$ which satisfies
	\begin{equation}
	\label{eq05}
	\langle a,\rho(\mu) \rangle <0.
	\end{equation}
\end{thm}
Again, this is in fact a corollary of a more general result given below. 
\begin{rem}
	\label{ref2rem1}
	Note that if $L$ is rational and $e_S(\Sigma_k)<\gamma_L(0)$, then (\ref{ref2eq2}) holds for every element $c$ in the lattice $\gamma_L(0)\cdot H^1(L;\Z) \leq H^1(L;\R)$:
	\[
	\mathcal{G}_L(c)=\langle [\omega]-\partial c,\pi_2(M,L)\rangle \subset \gamma_L(0)\cdot \Z-\langle c,\pi_1(L)\rangle \subset \gamma_L(0)\cdot \Z.
	\]
\end{rem}

\begin{rem}
	Theorem \ref{thm01} deduces information about a Hamiltonian system using Lagrange isotopies which \emph{cannot} be realized by globally defined symplectic isotopies. This should be compared with information about invariant measures coming from Lagrange isotopies induced by globally defined symplectic isotopies as in \cite{Polterovich14}. More precisely, the Lagrange isotopies $\psi$ which are the source of information in Theorem \ref{thm01} are those for which $\Flux_L(\psi)$ does \emph{not} lie in the image of the restriction map $H^1(M;\R) \to H^1(L;\R)$. To see this, note that, by the main result in Chekanov's beautiful paper \cite{Chekanov98}, no stable Lagrange isotopy with flux in the image of $H^1(M \times S^1;\R)\to H^1(L \times S^1;R)$ can take $L\times S^1$ into $\Sigma_k \times S^1$ if $e_S(\Sigma_k)<\gamma(L)$.
\end{rem}

\begin{rem}
	\label{ref2rem2}
	Let's explain how Theorem \ref{cor00} follows from Theorem \ref{thm01}. Consider an energy level $k\in \R$ such that $\Sigma_k$ is stably displaceable as well as a $\tilde{c}\in H^1(\Sigma_k;\R)$ such that $[\omega]|_{\pi_2(M,L)}=\partial \tilde{c}|_{\pi_2(M,L)}$.\footnote{By Lemma \ref{lem00} on page \pageref{lem00} such a $\tilde{c}$ exists if $(M,\omega)$ is weakly exact and $\pi_1(L)$ is Abelian.} If $a\in H^1(\Sigma_k;\R)$ satisfies (\ref{Meq5}) then $a|_L\in H^1(L;\R)$ can be written $a|_L=Tc'-\tilde{c}$ with $T\in (e_S(\Sigma_k),\infty)$ and $c'\in H^1(L;\Z)$. In particular 
	\[
	\mathcal{G}_L(-a|_L)=\langle [\omega]+\partial(Tc'-\tilde{c}),\pi_2(M,L)\rangle \subset T \langle c',\pi_1(L)\rangle \subset T\cdot \Z,
	\]
	so $\gamma_L(-a|_L)>e_S(\Sigma_k)$. Moreover, since $L\subset \Sigma_k$, we have $a \in \sh(\Sigma_k;L,-a|_L)$. To see this, consider the constant Lagrange isotopy $\psi:[0,1]\times L \to M$ given by $\psi_t(x)=x$. Then $\psi_1(L)=L\subset \Sigma_k$ and $\Flux_L(\psi)=0$, so
	\[
	\Flux_L(\psi)-(-a|_L)=a|_L=\psi_1^*a
	\]
	and $a \in \sh(\Sigma_k;L,-a|_L)$ as claimed. Hence, by Theorem \ref{thm01} there exists a $\mu \in \mathcal{M}(\Sigma_k;\phi_H)$ such that 
	\[
	\langle a, \rho(\mu) \rangle <0.
	\]
	This is essentially the statement of Theorem \ref{cor00} (the additional estimates are a consequence of Remark \ref{rem2} below).
\end{rem}

\begin{ex}
	\label{ex01}
	Let's illustrate the phenomenon captured in Theorem \ref{thm01} by a very simple example. Denote by $h:[0,\infty)\to \R$ the function whose graph is illustrated in Figure \ref{fig01} and consider the Hamiltonian $H(x,y)=h(\sqrt{x^2+y^2})$ on the plane. The symplectic structure we use on $\R^2(x,y)$ is $dx \wedge dy$.
	\begin{figure}[h]
		\centering
		\begin{tikzpicture}[scale=1]
\draw [thick, ->] (0,0) -- (4.5,0);
\draw [thick, ->] (0,0) -- (0,4);
\node [right] at (4.5,0) {\small $[0,\infty)$};
\node [left] at (0,2) {\small $2$};
\node [left] at (0,1) {\small $1$};
\node [below] at (3,0) {\small $3$};
\node [below] at (4,0) {\small $4$};
\node [below] at (1,0) {\small $1$};
\node [below] at (2,0) {\small $2$};
\node [right] at (3,1) {\small $\graph(h)$};
\node [left] at (0,3) {\small $\R$};
\draw [thick] (9,2) circle (2);
\node [right] at (10,0) {\small $\Sigma_{3/2}$};
\draw [thick] (9,2) circle (0.2);
\draw [red, dashed] (9,2) circle (1.15);
\draw [->] (9.5,2) arc (0:50:0.5cm);
\node [right] at (9.4,2.1) {\small $a$};
\draw [->] (10.5,2) arc (0:-20:1.3cm);
\node [right] at (10.4,1.9) {\small $b$};

\draw[scale=1,domain=1:4,smooth,variable=\x,black]  plot ({\x},{(\x-2)*(\x-2)});
\draw[scale=1,domain=0:1,smooth,variable=\x,black]  plot ({\x},{2-\x^2});

\end{tikzpicture}
		\caption{On the left $\graph(h)$ is indicated. On the right the sublevel set $\Sigma_{3/2}$ is indicated. The red dashed line indicates the level set $\{H=0\}$ which consists of fixed points for $\phi_H$. The arrows $a$ and $b$ indicate the direction of $\phi_H$-flowlines.}
		\label{fig01}
	\end{figure}
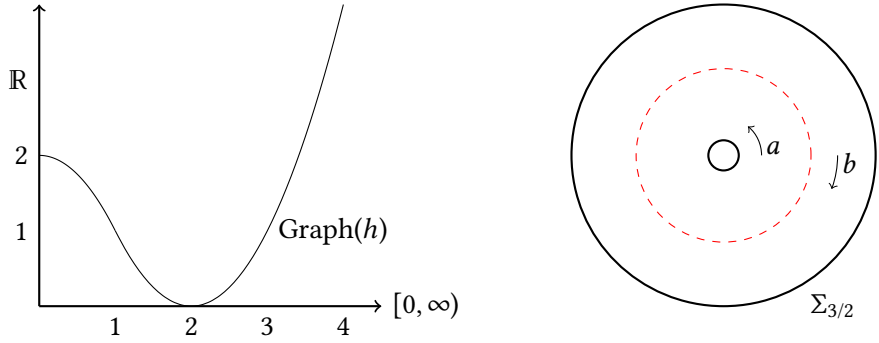
	The sublevel set $\Sigma_{3/2}$ is an annulus containing periodic orbits of $\phi_H$ which represent both positive and negative classes in $H_1(\Sigma_{3/2};\R)\cong \R$. This is exactly what our theory detects: We choose the Lagrangian $L:=\{x^2+y^2=16\}$, so that $e_S(\Sigma_{3/2})\leq e(\Sigma_{3/2})<16\pi=\gamma(L)$. Denote by $d\varphi$ the angle 1-form on $L$ (so that $\int_{l}d\varphi=2\pi$ for $l(t)=(4\cos(t),4\sin(t)),\ t\in [0,2\pi]$). According to (\ref{Meq4}), the Lagrange isotopy $\psi$ which "shrinks" $L$ to the Lagrangian $\{x^2+y^2=4\}=\{H=0\}\subset \Sigma_{3/2}$ satisfies
	\[
	\langle \Flux_{L}(\psi),[l] \rangle =16\pi-4\pi=12\pi, 
	\]
	so $\Flux_{L}(\psi)=6[d\varphi]\in H^1(L;\R)$. Hence, $6[d\varphi]\in \sh(\Sigma_{3/2};L)\subset H^1(\Sigma_{3/2};\R)$ and Theorem \ref{thm01} guarantees the existence of a $\mu \in \mathcal{M}(\Sigma_{3/2};\phi_H)$ such that 
	\[
	\langle 6[d\varphi],\rho(\mu) \rangle <0 \Leftrightarrow  \langle [d\varphi],\rho(\mu) \rangle <0,
	\]
	I.e. the motion detected by $\mu$ is that which is indicated by arrow $b$ in Figure \ref{fig01}. 
	Moreover, $8[d\varphi]\in \gamma(L)\cdot H^1(L;\Z)$ and
	\[
	-2[d\varphi]=\Flux_L(\psi)-8[d\varphi]\in \sh(\Sigma;L,8[d\varphi]),
	\]
	so Theorem \ref{thm01} also detects the existence of a $\nu \in \mathcal{M}(\Sigma_{3/2};\phi_H)$ such that 
	\[
	-\langle 2[d\varphi],\rho(\nu) \rangle <0 \Leftrightarrow \langle [d\varphi],\rho(\nu) \rangle >0.
	\] 
	I.e. the motion detected by $\nu$ is that which is indicated by arrow $a$ in Figure \ref{fig01}. For more sophisticated applications of the above results, see Section \ref{secexamples}. 
\end{ex}

The theorems above are in fact corollaries of a general theorem based on the study of a $\kappa$-function which we now define. Given a Hamiltonian $H\in C^{\infty}(M)$ which is bounded from below we associate to it a function $\kappa_{H:L}:H^1(L;\R)\to \R \cup \{+\infty \}$ defined by
\[
\kappa_{H:L}(c):=\inf_{\psi}\max_{\psi_1(L)}(H), \quad c\in H^1(L;\R),
\]
where the infimum runs over all Lagrange isotopies $\psi:[0,1]\times L\to M$ starting at $L$ and satisfying $\Flux_{L}(\psi)=c$. Again we use the convention that the infimum over $\emptyset$ equals $+\infty$. The motivation for this function comes from Aubry-Mather theory where it was discovered by Contreras-Iturriaga-Paternain \cite[Corollary 1]{ContrerasIturriagaPaternain98} that the Mather $\alpha$-function, associated to a convex Hamiltonian on the cotangent bundle of a closed manifold, can be defined in a similar way. 

\begin{thm}
	\label{thm00}
	Consider an autonomous Hamiltonian $H\in C^{\infty}(M;\R)$ which is proper and bounded from below. Let $c\in H^1(L;\R)$ and suppose $k\in \R$ is an energy value such that $k\leq \kappa_{H:L}(c)$. Then for every
	\[
	a\in \sh(\Sigma_k;L,c)
	\]
	there exists an ergodic $\mu \in \mathcal{M}(\Sigma_k;\phi_H)$ which satisfies 
	\begin{equation}
		\label{eq06}
		\langle a,\rho(\mu) \rangle <0.
	\end{equation}
\end{thm}

\begin{rem}
	\label{rem2}
	In fact the proof of Theorem \ref{thm00} (presented in Section \ref{secproofthm00}) provides information both about $\langle a,\rho(\mu) \rangle$ as well as $\supp(\mu)$. More precisely, if $\psi:[0,1]\times L\to (M,\omega)$ is a Lagrange isotopy with $\Flux_{L}(\psi)-c=\psi_1^*a$ and $\psi_1(L)\subset \Sigma_k$ then for a given $\epsilon>0$ one can find an ergodic $\mu \in \mathcal{M}(\phi_H)$ achieving  
	\begin{align*}
		\langle a,\rho(\mu) \rangle &\leq \max_{\psi_1(L)}(H)-k+\epsilon \\
		 H(\supp(\mu))&\subset [\max_{\psi_1(L)}(H)+\tfrac{\epsilon}{2},k-\tfrac{\epsilon}{2}].
	\end{align*}
\end{rem}

The proof of Theorem \ref{thm00} relies heavily on beautiful ideas due to Buhovsky-Entov-Polterovich \cite{BuhovskyEntovPolterovich12}, Entov-Polterovich \cite{EntovPolterovich17} and Polterovich \cite{Polterovich14}. Of course this result is only useful if we can find reasonable lower bounds for $\kappa_{H:L}$. The main source of such estimates come from rigidity results in symplectic topology, which is where $\mathcal{G}_L(c)$ from (\ref{eqrev2eq1}) comes into play. This happens for the following reason: If $\psi:[0,1]\times L\to M$ is a Lagrange isotopy starting at $L$ with $\Flux_L(\psi)=c\in H^1(L;\R)$, then the period group of $\psi_1(L)\subset (M,\omega)$ is exactly $\mathcal{G}_L(c)$:
\begin{equation}
	\label{eq33}
\langle [\omega], \pi_2(M,\psi_1(L))\rangle =\mathcal{G}_L(c).
\end{equation}

The following Proposition was the original idea employed in \cite{EntovPolterovich17}. For more quantitative versions of this idea, see below.
\begin{prop}[\cite{EntovPolterovich17}]
	Suppose $c\in H^1(L;\R)$ satisfies $\partial c|_{\pi_2(M,L)}=[\omega]|_{\pi_2(M,L)}$ and suppose that $(M,\omega)$ does not admit any weakly exact Lagrangian submanifolds. Then $\kappa_{H,L}(c)=+\infty$ for all Hamiltonians $H$.
\end{prop}
\begin{proof}
	There exists no Lagrange isotopy $\psi$ starting at $L$ with $\Flux_L(\psi)=c$. This follows from (\ref{eq33}) and our assumption.
\end{proof}

This proposition can be refined in several ways. Suppose $(M,N, \omega)$ is a \emph{subcritical polarized K{\"a}hler manifolds} in the sense of \cite{BiranCieliebak01}. I.e. $(M,\omega)$ is a closed K{\"a}hler manifold and $N\subset M$ is a subset such that $(M\backslash N,\omega)$ admits the structure of a subcritical Stein manifold. In this situation we have the following estimate of $\kappa_{H:L}$.
\begin{prop}
	Let $(M,N, \omega)$ be as above and suppose $H\in C^{\infty}(M)$. If $L\subset (M,\omega)$ is a closed Lagrangian and $c\in H^1(L;\R)$ satisfies the condition that $\partial c|_{\pi_2(M,L)}=[\omega]|_{\pi_2(M,L)}$ then
	\begin{equation*}
	\kappa_{H:L}(c)\geq \min_{N}H.
	\end{equation*}
\end{prop}  
\begin{proof}
	Suppose $\psi:[0,1]\times L \to (M,\omega)$ is a smooth Lagrange isotopy starting at $L$ with $\Flux_L(\psi)=c\in H^1(L;\R)$. It suffices to show that $\psi_1(L)\cap N\neq \emptyset$. Suppose for contradiction that $\psi_1(L)\subset (M \backslash N,\omega)$. Since $(M \backslash N,\omega)$ is subcritical Stein it follows from \cite[Lemma 3.2]{BiranCieliebak02} that $\psi_1(L)$ can be displaced by a compactly supported Hamiltonian diffeomorphism.  However, $\mathcal{G}_L(c)=\{0\}$, which means that $\psi_1(L)$ is weakly exact, so it cannot be displaced by a Hamiltonian diffeomorphism \cite{Polterovich93}. This finishes the proof.
\end{proof}
The following Proposition connects Theorem \ref{thm00} to the main result in \cite{Polterovich14}.
\begin{prop}
	Suppose that $M$ is closed, that the restriction map $r_L:H^1(M;\R)\to H^1(L;\R)$ is surjective and that $L$ is non-displaceable. I.e. $\phi(L)\cap L\neq \emptyset$ for all $\phi \in \Ham(M,\omega)$. Then 
	\begin{equation}
	\label{eq34}
	\kappa_{H:L}(0)\geq \min_L H \quad \forall \ H\in C^{\infty}(M).
	\end{equation}
\end{prop}
\begin{proof}
	Let $\psi:L \times [0,1]\to M$ be a Lagrange isotopy starting at $L$ with $\Flux_L(\psi)=0\in H^1(L;\R)$. By \cite[Lemma 6.6]{Solomon13} there exists a symplectic isotopy $\theta:M\times [0,1]\to M$ with $\theta_0=\id$ such that $\theta_t(L)=\psi_t(L)$ for all $t\in [0,1]$ and such that $\Flux(\theta)=0\in H^1(M;\R)$. In particular $\theta_1\in \Ham(M,\omega)$. Hence, $\psi_1(L)\cap L=\theta_1(L)\cap L\neq \emptyset$ which implies (\ref{eq34}).
\end{proof}

The guiding philosophy above is that symplectic rigidity prevents a given Lagrangian from being (Lagrangian) isotoped with a given (Lagrange) flux into a fixed sublevel set of our Hamiltonian $H$. Perhaps the most fundamental situation where this idea can be used to estimate $\kappa_{H:L}$ from below is the following:
\begin{prop}
	\label{prop01}
	Let $H\in C^{\infty}(M;\R)$ be bounded from below. Suppose $c\in H^1(L;\R)$ and $k\in \R$ meet the condition 
	\begin{equation}
		\label{eq000}
		e_S(\Sigma_k)<\gamma_L(c).
	\end{equation}
	Then
	\[
	\kappa_{H:L}(c)\geq k.
	\]
\end{prop}
\begin{proof}[Proof of Proposition \ref{prop01}]
	Consider a Lagrange isotopy $\psi:[0,1]\times L\to M$ starting at $L$ and satisfying 
	\begin{equation}
		\label{eq00}
		\Flux_L(\psi)=c.
	\end{equation}
	We need to show that $\max_{\psi_1(L)}(H)\geq k$. Since $S^1\subset T^*S^1$ is exact, $\mathcal{G}_L(c)$ coincides with the period group of $\psi_1(L)\times S^1 \subset (M\times T^*S^1,\omega \oplus \omega_0)$. Hence, a fundamental result first due to Polterovich \cite{Polterovich93} and later Chekanov \cite{Chekanov98} says that $\gamma_L(c) \leq e(\psi_1(L)\times S^1)=e_S(\psi_1(L))$ (i.e. $\psi_1(L)$ is stably non-displaceable if $\gamma_L(c)=\infty$). It now follows from (\ref{eq000}), that $e_S(\Sigma_k)<e_S(\psi_1(L))$ which in turn gives 
	\[
	\psi_1(L)\times S^1 \nsubseteq \Sigma_k\times S^1,
	\]
	or $\psi_1(L) \nsubseteq \Sigma_k=\{H<k\}$. Thus, $\max_{\psi_1(L)}(H)\geq k$ which finishes the proof.
\end{proof}
\begin{rem}
	Together Theorem \ref{thm00} and Proposition \ref{prop01} establish a link between the fundamental notion of \emph{displacement energy} in symplectic topology and the fundamental idea of looking for invariant measures with a prescribed rotation vector in Aubry-Mather theory. We hope to explore this idea further in future research. More precisely, we hope to find conditions where one can make sense of thinking of the measures in Theorem \ref{thm00} as Mather measures and derive more precise properties of these measures. Other results linking the idea of Hofer geometry/displacement energy with Aubry-Mather theory known to the author are \cite{Siburg98} and \cite{SorrentinoViterbo10}.
\end{rem}

\begin{proof}[Proof of Theorem \ref{thm01}]
	Recall that we consider an energy value $k\in \R$ and a class $c\in H^1(L;\R)$ such that $e_S(\Sigma_k)<\gamma_L(c)$. By Theorem \ref{thm00} it suffices to show that $\kappa_{H:L}(c)\geq k$. This is the content of Proposition \ref{prop01}. Note that in combination with Remark \ref{ref2rem2} on page \pageref{ref2rem2} this in fact also proves Theorem \ref{cor00}.
\end{proof}

\subsection{Examples}
\label{secexamples}

\begin{ex}
	\label{ex02}
	Consider $\R^{2n}$ equipped with the standard symplectic structure $\omega:=\sum_{l=1}^n dx_l\wedge dy_l$ and denote by $p:\R^{2n}\to \R^2$ the projection onto the first factor of $\R^{2n}=\R^2\times \cdots \times \R^2$. Choose a Hamiltonian $H\in C^{\infty}(\R^{2n})$ which is proper and bounded from below and fix an energy value $k\in \R$ such that $\Sigma_k \neq \emptyset$. Choose $r>0$ such that $p(\Sigma_k)\subset \R^2$ is contained in the open disc of radius $r$ centered at the origin. Denote by $L\subset (\R^{2n},\omega)$ the split Lagrangian $L=S^1(r)\times \cdots \times S^1(r)$ which is the product of circles in $\R^2$ of radius $r$, centered at the origin. Then $e_S(\Sigma_k)\leq e(\Sigma_k)<r^2\pi=\gamma(L)$, so Theorem \ref{thm01} (in combination with Remark \ref{ref2rem1}) gives the following statement: For every $c\in r^2\pi \cdot H^1(L;\Z)\subset H^1(L;\R)$ and every $a\in \sh(\Sigma_k;L,c)$ there exists $\mu \in \mathcal{M}(\Sigma_k;\phi_H)$ such that 
	\[
	\langle a,\rho(\mu) \rangle <0.
	\] 
\end{ex}

\begin{rem}
Suppose in the above example that $n=2$ and that $\Sigma_k$ contains a Lagrangian $2$-torus $L'\subset \Sigma_k$ such that the restriction map $H^1(\Sigma_k;\R)\to H^1(L';\R)$ is surjective. Then it follows from \cite[Theorem A]{RizellGoodmanIvrii16} that there exists a Lagrange isotopy $\psi:[0,1]\times L\to M$ such that $\psi_1(L)=L'$. In particular $\sh(\Sigma_k;L,c)\neq \emptyset$ for all $c\in H^1(L;\R)$. Of course in this case, one could also apply Theorem \ref{cor00} instead of Theorem \ref{thm01}.
\end{rem}
Apart from $\R^{2n}$, the main example we have in mind is that of Hamiltonian dynamics on a twisted cotangent bundle.

\begin{ex}
	Consider the two-torus $\mathbb{T}^2=\R^2 / \Z^2$ and denote by $\Omega$ a closed $2$-form on $\mathbb{T}^2$ such that $\int_{\mathbb{T}^2} \Omega=1$. The symplectic structure on the \emph{twisted cotangent bundle} $(T^*\mathbb{T}^2,\omega)$ is defined by 
	\[
	\omega:=d\lambda_{\mathbb{T}^2} +\pi^*\Omega,
	\] 
	where $\pi:T^*\mathbb{T}^2\to \mathbb{T}^2$ denotes the footpoint map and $\lambda_{\mathbb{T}^2}$ denotes the Liouville $1$-form. Note that $(T^*\mathbb{T}^2,\omega)$ is weakly exact. Let $H\in C^{\infty}(T^*\mathbb{T}^2)$ be a Hamiltonian which is proper and bounded from below. By \cite{GinzburgKerman99} every sublevel set $\Sigma_k=\{H<k\}$ is displaceable. We want to apply this fact to study the dynamics of $\phi_H|_{\Sigma_k}$. Suppose $L\subset \Sigma_k \subset (T^*\mathbb{T}^2,\omega)$ is a Lagrangian 2-torus.\footnote{There are several ways of finding Lagrangian 2-tori in $(T^*\mathbb{T}^2,\omega)$. One can find "small" ones in Darboux charts but there are also several other approaches. See e.g. Example \ref{ex03}. Note that, since every closed Lagrangian in $(T^*\mathbb{T}^2,\omega)$ is displaceable, it follows from the adjunction formula, that every closed orientable Lagrangian in $(T^*\mathbb{T}^2,\omega)$ is topologically a 2-torus.} Since $\pi_1(L)$ is Abelian, Lemma \ref{lem00} on page \pageref{lem00} guarantees the existence of a $\tilde{c}\in H^1(L;\R)$ such that $\partial \tilde{c}|_{\pi_2(M,L)}\equiv [\omega]|_{\pi_2(M,L)}$. Applying Theorem \ref{cor00} we then conclude that for any $a\in H^1(\Sigma_k;\R)$ satisfying 
	\[
	a|_L\in (e_S(\Sigma_k),\infty)\cdot H^1(L;\Z)-\tilde{c}
	\]
	and any $\epsilon>0$ there exists a $\mu \in \mathcal{M}(\Sigma_k;\phi_H)$ such that 
	\[
	\langle a,\rho(\mu) \rangle \leq \max_L H-k+\epsilon.
	\]
\end{ex}

\begin{ex}
	\label{ex03}
	Consider the 2-sphere $S^2$ equipped with a closed $2$-form $\Omega$ such that $\int_{S^2}\Omega=1$. Denote by $\pi:T^*S^2 \to S^2$ the footpoint map. We want to apply our results above to study Hamiltonian systems on the twisted cotangent bundle $(T^*S^2,\omega)$, where 
	\[
	\omega:=d\lambda_{S^2} + \pi^*\Omega.
	\] 
	Clearly $(T^*S^2,\omega)$ is \emph{not} weakly exact. There exists a rational Lagrangian 2-torus $L\subset (T^*S^2,\omega)$ with $\gamma_L(0)=1$. Assuming this fact for now, let $H\in C^{\infty}(T^*S^2)$ denote an autonomous Hamiltonian which is proper and bounded from below and let $k\in \R$ denote an energy value such that $e_S(\Sigma_k)<1$. Then by Theorem \ref{thm01} (and the following remark), for every $c\in H^1(L;\Z)$ and every $a\in \sh(\Sigma_k;L,c)$ there exists a $\mu \in \mathcal{M}(\Sigma_k;\phi_H)$ such that
	\[
	\langle a,\rho(\mu) \rangle <0.
	\] 
	For concreteness, consider $H(q,p)=\frac{|p|^2}{2}+U(q)$ for some Riemannian metric on $T^*S^2\to S^2$ and $U\in C^{\infty}(S^2)$. By \cite[Proposition 1.4]{Schlenk06} there exists a $\delta>0$ such that $e_S(\{(q,p)\ |\ |p|\leq \delta\})<1$. In particular we conclude that, for all sufficiently small $k>\min(H)=\min(U)$, we have $e_S(\Sigma_k)<1$. Of course, if $U$ is constant then $\Sigma_k$ is simply connected for all $k\in \R$, but this is not the case if $U$ is not constant. 
	
	In order to see the existence of $L$, let's be explicit and view 
	\[
	S^2=\{(x,y,z)\in \R^3 \ |\ x^2+y^2+z^2=1\}.
	\]
	Since $\Omega|_{S^2 \cap \{z>-1\}}$ is exact we can denote by $\xi$ a $1$-form on $S^2 \cap \{z>-1\}$ such that $\Omega=d\xi$ and choose a smooth function $\rho:S^2\to [0,1]$ such that 
	\[
	\rho(x,y,z)=
	\begin{cases}
	1, & \text{if}\ z\geq -1/2 \\
	0,& \text{if}\ z\leq -2/3.
	\end{cases}
	\]
	By Moser's argument the symplectic manifold $(T^*S^2,\widetilde{\omega})$, where 
	\[
	\widetilde{\omega}=\omega - d\pi^*(\rho \xi),
	\]
	is symplectomorphic to $(T^*S^2,\omega)$. Denote by $f:B^2(2)\to S^2$ an embedding which identifies the closed $2$-disc $B^2(2):=\{(x_1,x_2)\in \R^2 \ |\ x_1^2+x_2^2 \leq 4\}$ with the upper hemisphere $S^2 \cap \{z \geq 0\}$. Since $\widetilde{\omega}|_{S^2 \cap \{z \geq 0\}}=d\lambda_{S^2}|_{S^2 \cap \{z \geq 0\}}$ $f$ induces a symplectic embedding 
	\[
	F=(df^*)^{-1}:(T^*B^2(2),\lambda_{B^2(2)}) \to (T^*S^2,\widetilde{\omega})
	\]
	Denote by $T\subset B^4(2)\subset B^2(2)\times \R^2=T^2B^2(2)$ a rational Lagrangian $2$-torus with $\gamma(T)=1$. We claim that $F(T)\subset (T^*S^2,\widetilde{\omega})$ too is a rational Larangian $2$-torus with $\gamma(F(T))=1$. In order to see this, let $u:(B^2(1),\partial B^2(1))\to (T^*S^2,F(T))$ denote a smooth topological disc on $F(T)$. It suffices to check that $\int_u \widetilde{\omega} \in \Z$. Since we clearly have $\int_u d\lambda_{S^2}\in \Z$ it in fact suffices to check that $\int_{v} (\Omega-d(\rho \xi)) \in \Z$, where $v=\pi \circ u$. To see this we note that, since $\pi(F(T))\subset S^2 \cap \{z\geq 0 \}$ and $S^2 \cap \{z\geq 0 \}$ is contractible, we can extend $v$ to a map $\widetilde{v}:B^2(2) \to S^2$ such that $\widetilde{v}(\partial B^2(2))=\{(0,0,1)\}$. Since the $2$-form $(\Omega-d(\rho \xi))$ vanishes on $S^2 \cap \{z\geq 0\}$ we have $\int_{\widetilde{v}}(\Omega-d(\rho \xi))=\int_{v}(\Omega-d(\rho \xi))$ and since $\widetilde{v}(\partial B^2(2))=\{(0,0,1)\}$, $\widetilde{v}$ induces a map $\widehat{v}:S^2 \to S^2$. Now $\int_{\widetilde{v}}(\Omega-d(\rho \xi))$ is simply the degree of $\widehat{v}$. In particular it is an integer. This proves the claim that $F(T)\subset (T^*S^2,\widetilde{\omega})$ is a rational Lagrangian torus with $\gamma(F(T))=1$. Since $(T^*S^2,\widetilde{\omega})\cong (T^*S^2,\omega)$ we conclude that there exists a rational Lagrangian torus $L\subset (T^*S^2,\omega)$ with $\gamma(L)=1$. 
\end{ex}

\begin{rem}
	Of course the above examples can also be carried out for twisted cotangent bundles of closed manifolds $N\neq \T^2,S^2$. The main point is that if $[\omega]\neq 0$ then $N\subset (T^*N,\omega)$ has stable displacement energy 0 by \cite[Proposition 1.4]{Schlenk06}, so if $H\in C^{\infty}(T^*N)$ is \emph{mechanical}\footnote{I.e. $H(q,p)=\frac{|p|^2}{2}+U(q)$ for some Riemannian metric and $U\in C^{\infty}(N)$.} then $\Sigma_k$ has small displacement energy for all small enough $k>\min(H)$. However, the choice of $S^2$ in the above example is not coincidental: One could ask if it is possible to study Hamiltonian systems on the twisted $T^*S^2$ from the previous example using Lagrangian Floer homology. However, this approach seems unlikely to succeed. Indeed, as pointed out to me by MathOverflow-user Nikolaki \cite{Nikolaki}, if one manages to find a Lagrangian $L\subset (T^*S^2,\omega)$ for which there exists a well-defined Floer homology $HF(L)$, then probably $HF(L)=0$. This follows from the fact that (if well-defined) $HF(L)$ should be a module over the symplectic homology $SH(T^*S^2)$ of $(T^*S^2,\omega)$ and a result due to Benedetti and Ritter \cite[Theorem 1.1]{BenedettiRitter18} says that $SH(T^*S^2)=0$ (see also	\cite{Ritter14} as well as Benedetti's thesis \cite{Benedetti15} for applications of symplectic homology to Hamiltonian systems on twisted cotangent bundles).
\end{rem}

\section{Preliminaries}
\label{seqprem}
Here we discuss some of the background material used frequently in the above sections. Recall that a Lagrangian submanifold $L\subset (M,\omega)$ is said to be \emph{monotone} if there exists a constant $\tau_L\geq 0$ such that 
\[
\omega|_{\pi_2(M,L)}=\tau_L\cdot \mu|_{\pi_2(M,L)},
\]
where $\omega|_{\pi_2(M,L)}$ denotes integration of $\omega$ and $\mu|_{\pi_2(M,L)}$ denotes the Maslov index. If in addition the minimal Maslov number of $L$ is $\geq 2$, then the quantum homology $QH_*(L;\Z_2)$ and Floer homology $HF(L;\Z_2)$ of $L$ are well-defined \cite{BiranCornea07}, \cite{Zapolsky15}. Our main reference for $QH_*(L;\Z_2)$ and $HF_*(L;\Z_2)$ is Zapolsky's excellent \cite{Zapolsky15}. In order to set the notation we use here we will discuss a few details about $HF_*(L;\Z_2)$ and the construction of spectral invariants. Recall that, given a Hamiltonian $H\in \mathcal{H}$ such that $\phi_H^1(L)\pitchfork L$ and a generic path $\{J_t\}_{t\in [0,1]}$ of $\omega$-compatible almost complex structures, $HF_*(L{:}H,J)$ is the homology of a Morse-type chain complex $(CF_*(L{:}H,J),d)$ generated by critical points of the action functional 
\[
\mathcal{A}_{H:L}(\widetilde{\gamma}=[\gamma,\widehat{\gamma}])=\int_0^1 H_t(\gamma(t))\ dt-\int \widehat{\gamma}^*\omega.
\]
Here the input $\widetilde{\gamma}$ consists of a smooth path $\gamma{:}([0,1],\{0,1\})\to (M,L)$ as well as an equivalence class of \emph{cappings} $\widehat{\gamma}$ of $\gamma$ in the sense of \cite{Zapolsky15}. The equivalence relation identifies cappings which have the same symplectic area.
As a result, generators of $CF_*(L{:}H,J)$ have an associated action and we denote by $(CF^a_*(L{:}H,J),d)$ $(a\in \R)$ the subcomplex generated by elements whose action is $<a$. This is indeed a subcomplex as $d$ decreases action, and its homology is denoted by $HF^a_*(L{:}H,J)$. The inclusion $i_a:CF^a_*(L{:}H,J)\to CF_*(L{:}H,J)$ induces maps $i_*^a:HF_*^a(L{:}H,J)\to HF_*(L{:}H,J)$. By standard arguments $HF_*(L{:}H,J)$ is independent of the data $(H,J)$ and identifying all such groups via canonical continuation isomorphisms we obtain a ring $HF(L;\Z_2)$, which is isomorphic to $QH(L;\Z_2)$ via a $\PSS$-isomorphism 
\[
\PSS:QH_*(L;\Z_2)\stackrel{\cong}{\to}HF_*(L;\Z_2).
\]
Under the assumption $QH_*(L;\Z_2)\neq 0$, the Leclercq-Zapolsky spectral invariant $l_L:\widetilde{\Ham}(M,\omega)\to \R$ which we consider is defined by 
\[
l_L(\tilde{\phi}_H):=\inf\{a\in \R \ |\ \PSS([L])\in \Image(i^a_*)\subset HF_*(L;\Z_2)  \},
\] 
where $i_*^a:HF_*^a(L{:}H,J)\to HF_*(L{:}H,J)\cong HF_*(L;\Z_2)$ and $[L]\in QH(L;\Z_2)$ denotes the "fundamental class", i.e. the unity for the ring-structure on $QH_*(L;\Z_2)$. For further details on $l_L$ and the properties satisfied by $l_L$ we refer to \cite{LeclercqZapolsky15} and \cite{ZapolskyMonznerVichery12}. An important property for us is that $l_L$ satisfies the \emph{triangle inequality}:
\[
l_L(\tilde{\phi}_H\tilde{\phi}_K)\leq l_L(\tilde{\phi}_H)+l_L(\tilde{\phi}_K) \quad \forall \ H,K\in \mathcal{H}.
\]
As a consequence of this, the limit (\ref{eq2}) exists and $\sigma_{H:L}$ is well-defined.
The following theorem is due to \cite{ZapolskyMonznerVichery12} in a slightly different setting. In fact, \cite{ZapolskyMonznerVichery12} presents many more properties of $\sigma_{H:L}$ in the setting of a cotangent bundle. Here we only list the properties which will be useful to us. 
\begin{thm}[\cite{ZapolskyMonznerVichery12}]
	\label{prop1}
	Given any $H\in \mathcal{H}$, the function $\sigma_{H:L}:H^1(M;\R)\to \R$ from Definition \ref{def2} is well-defined and it satisfies the following properties 
	\begin{enumerate}[a)]
		\item 
		$\sigma_{H:L}$ is locally Lipschitz with respect to any norm on $H^1(M;\R)$. In particular $\sigma_{H:L}$ is differentiable almost everywhere.
		\item
		Suppose $\psi=\{\psi_t \}_{t\in [0,1]}$ is a symplectic isotopy with $\psi_0=\id$. Then 
		\[
		\min_{x\in \psi_1(L)}H(x)\leq \sigma_{H:L}(c)\leq \max_{x\in \psi_1(L)}H(x),
		\]
		where $c:=\Flux(\psi)\in H^1(M;\R)$.
		\item 
		If $\psi=\{\psi_t\}_{t\in [0,1]} \subset \Symp(M,\omega )$ is a path of symplectomorphisms with $\psi_0=\id$ and $c':=\Flux(\psi)\in H^1(M;\R)$ then 
		\[
		\sigma_{H:\psi_1(L)}(c)=\sigma_{H:L}(c+c') \quad \forall \ c\in H^1(M;\R).
		\]
		\item
		For every $\phi \in \Ham(M,\omega)$ we have $\sigma_{H:L}=\sigma_{H:\phi(L)}$. 
	\end{enumerate}
\end{thm}
In \cite{Bisgaard16} we obtained an alternative version of the last property: If $L'\subset (M,\omega)$ is another monotone Lagrangian submanifold which is monotone Lagrangian cobordant to $L$ then $\sigma_{H:L}=\sigma_{H:L'}$.
We now make a few remarks concerning the definition of $\sigma_{H:L}$ from Section \ref{noncompact}. For this purpose, let $(M,\omega=d\lambda)$ be a Liouville manifold with ideal contact boundary $(\Sigma,\lambda_0)$ (see page \pageref{Mcondb}). We first prove Proposition \ref{Mprop1}.
\begin{proof}[Proof of Proposition \ref{Mprop1}]
Fix a closed 1-form $\eta_0$ on $M$ satisfying $[\eta_0]=c\in H^1_{dR}(M;\R)$. Define $\eta_1$ on $(0,\infty)\times \Sigma$ as the pull-back of $\eta_0|_{\{1\}\times \Sigma}$ via the projection $(0,\infty)\times \Sigma \to \Sigma$. By homotopy invariance $[\eta_0|_{(0,\infty)\times \Sigma}]=[\eta_1]$ in $H^1_{dR}((0,\infty)\times \Sigma;\R)$, so there exists $f\in C^{\infty}((0,\infty)\times \Sigma)$ such that $\eta_1=\eta_0|_{(0,\infty)\times \Sigma}+df$. Choose a function $\chi \in C^{\infty}((0,\infty))$ such that $\chi(s)=0$ for all $s\in (0,\tfrac{1}{4})$ and $\chi(s)=1$ for all $s\in (\tfrac{3}{4},\infty)$. Then 
\[
\eta:=\eta_0+d(\chi f)
\]
defines a closed 1-form on all of $M$ in the class $c$. Since $\eta=\eta_1$ on $(\tfrac{3}{4},\infty)\times \Sigma$ we have the wanted property:
\begin{equation}
\label{Meq2}
\eta|_{\{s\}\times \Sigma}=\eta|_{\{1\}\times \Sigma} \quad \forall \ s\in [1,\infty).
\end{equation}
An easy computation using that $\omega=ds\wedge \lambda_0+sd\lambda_0$ on $(0,\infty)\times \Sigma$ shows that, if $\eta$ satisfies (\ref{Meq2}), then the vectorfield $X\in \mathfrak{X}(M)$ characterized by $\eta =i_X\omega$ satisfies $ds(X)|_{(s,x)}=\eta_{(s,x)}(R(x))=\eta_{(1,x)}(R(x))$ for all $(s,x)\in [1,\infty)\times \Sigma$ (here $R$ denotes the Reeb vector field on $(\Sigma,\lambda_0)$). Since $\Sigma$ is compact, this implies that $X$ is complete.
\end{proof}

Suppose now that $\eta_1$ and $\eta_2$ are two closed 1-forms in the class $c$ satisfying (\ref{Meq2}) for (potentially) different ideal contact boundaries $(\Sigma_1,\lambda_0^1)$ and $(\Sigma_2,\lambda_0^2)$ of $(M,d\lambda)$. Denote by $(\psi^k_t)_{t\in \R}$ the flows generated by the vectorfields $X_k\in \mathfrak{X}(M)$ characterized by $i_{X_k}\omega=\eta_k$, $k=1,2$. The symplectic isotopy $t\mapsto \psi_{-t}^1\psi^2_t$ is generated by the time-dependent vectorfield $(\psi^1_{-t})_*X_2-X_1$. Since
\[
[i_{(\psi^1_{-t})_*X_2-X_1}\omega]=[i_{(\psi^1_{-t})_*X_2}\omega]-[\eta_1]=[(\psi_t^1)^*\eta_2]-[\eta_1]=c-c=0
\]
it follows that $\psi_{-t}^1\psi^2_t=\phi_t$ for a (possibly non-compactly supported) Hamiltonian isotopy $t\mapsto \phi_t$. Since $\psi_1^2=\psi^1_1\phi_1$ it follows from arguments completely analogous to those in \cite{ZapolskyMonznerVichery12} that for $H\in \mathcal{H}$ we have 
\begin{equation}
\label{Meq3}
\lim_{\N \ni k\to \infty}\frac{l_L((\psi^1_1)^{-1}\widetilde{\phi}_H^k\psi^1_1)}{k}=\lim_{\N \ni k\to \infty}\frac{l_L((\psi^2_1)^{-1}\widetilde{\phi}_H^k\psi^2_1)}{k}.
\end{equation}
The crux of this argument (like in \cite{ZapolskyMonznerVichery12}) lies in the fact that, since $\mathcal{O}_+(L)$ is compact, all Hamiltonians involved in the above computation can be cut-off outside a compact set without changing the spectral invariants. Hence, ultimately one only has to work with compactly supported Hamiltonians (for which (\ref{Meq3}) is standard). This concludes the reasoning that $\sigma_{H:L}(c)$ is well-defined in the non-compact setting. Moreover, it is Lipschitz by exactly the same arguments as in \cite{ZapolskyMonznerVichery12}, so it has a non-empty set of Clarke subdifferentials at every point $c\in H^1(M;\R)$.

\subsection{Subdifferentials and rotation vectors}
\label{secmeas}
In this section we recall a few basic facts about the Clarke subdifferential $\partial f$ of a function $f$ as well as about the approximate subdifferential $\partial_A f$ which is needed in the proofs below. For further references on the Clarke subdifferential we refer to \cite{Clarke83} and for further references on the approximate subdifferential we refer to \cite{Ioffe84}. Denote by $V$ a finite dimensional vector space over $\R$ and by $f:V\to \R$ a function which is locally Lipschitz (with respect to any norm on $V$). The generalized directional derivative $f^{\circ}(x,v)$ of $f$ at $x\in V$ in direction $v\in V$ is by definition the number
\[
f^{\circ}(x,v):=\limsup_{\substack{y\to x \\ \tau \downarrow 0}}\frac{f(y+\tau v)-f(y)}{\tau}. 
\]
$f^{\circ}(x,v)$ is a real number because $f$ is locally Lipschitz. The \emph{Clarke subdifferential} (or \emph{generalized derivative}) of $f$ at $x\in V$ is by definition the non-empty subset 
\[
\partial f(x)=\{v^*\in V^*\ |\  \langle v^*,v\rangle \leq f^{\circ }(x,v) \ \forall \ v\in V\}
\] 
of $V^*$. Note that since $f$ is Locally Lipschitz it is differentiable almost everywhere.\footnote{This should be understood in the following sense: If we identify $V\cong \R^k$ by choosing a basis, then the set of points in $\R^k$ at which $f$ isn't differentiable is a Lebesgue 0-set.}
A key property of $\partial f$ is \cite[Theorem 2.5.1]{Clarke83}, according to which, for any Lebesgue 0-set $S\subset V$ containing the non-differentiability points of $f$ we have
\begin{equation}
\label{eq99}
\partial f(x)=\conv( \{\lim_{\varsigma \to \infty}df(x_{\varsigma})\ |\  (x_{\varsigma})_{\varsigma \in \N}\subset V\backslash S \ \text{s.t.} \ x_{\varsigma}\to x \} ),
\end{equation}
where $\conv(X)\subset V$ denotes the convex hull of the subset $X\subset V$. I.e. $\partial f(x)$ is the convex hull of the set of all limit points of sequences of differentials $df(x_{\varsigma})$ of $f$ with $(x_{\varsigma})_{\varsigma \in \N}\subset V\backslash S$ a sequence such that $x_{\varsigma}\to x$. Another important property which will be used below is Lebourg's mean value theorem for $\partial f$ \cite[Theorem 2.3.7]{Clarke83}: For every pair $x,y\in V$ there exists a $t\in (0,1)$ such that for some 
\[
v^*\in \partial f(ty+(1-t)x)
\]
we have
\[
\langle v^*,x-y \rangle =f(x)-f(y).
\] 

We will now discuss some of the basic properties of the approximate subdifferential $\partial_Af$ of $f$ due to Ioffe \cite{Ioffe84}. We point out that the $\partial_Af$ makes sense also if $f$ is not locally Lipschitz, but here we restrict ourselves to considering the case when $f$ is locally Lipschitz. For $x,v\in V$, the \emph{lower Dini directional derivative} of $f$ at $x$ in direction $v$ is defined by 
\[
d^-f(x;v):=\liminf_{\tau \downarrow 0}\frac{f(x+\tau v)-f(x)}{\tau }\in \R
\]
and the \emph{Dini subdifferential} $\partial^-f(x)\subset V^*$ of $f$ at $x$ is the subset 
\begin{equation}
\label{eq1}
\partial ^-f(x):=\{ v^*\in V^*\ |\ \langle v^*,v\rangle \leq d^-f(x;v) \ \forall \ v\in V \}
\end{equation}
of $V^*$. Finally the approximate subdifferential of $f$ at $x$ is defined as the subset\footnote{$||\cdot ||$ is any norm on $V$.}
\begin{equation}
\label{eq222}
\partial_Af(x):=\bigcap_{\delta>0}\bigcup_{\ ||z-x||<\delta}\partial^-f(z), 
\end{equation}
of $V^*$. A key property of $\partial_Af$ (when $f$ is locally Lipschitz) is \cite[Theorem 2]{Ioffe84}, according to which the Clarke subdifferential and the approximate subdifferential are related by 
\begin{equation}
\label{eq111}
\partial f(x)=\overline{\conv(\partial_Af(x))} \quad \forall \ x\in V.
\end{equation}
Our main interest in $\partial_A$ is that it behaves well under uniform approximation: By a result due to Jourani \cite[Theorem 3.2]{Jourani99}, if $(f_{\varsigma})_{\varsigma \in \N}$ is a sequence of locally Lipschitz functions $V\to \R$ such that 
\[
f_{\varsigma} \stackrel{\varsigma \to \infty}{\longrightarrow}f
\]
uniformly, then 
\begin{equation}
\label{eq333}
\partial_Af(x)\subset \limsup_{\substack{ \varsigma \to \infty \\ y\to x}}\partial_Af_{\varsigma}(y) \quad \forall \ x\in V,
\end{equation}
where 
\[
\limsup_{\substack{\varsigma \to \infty\\ y\to x}}\partial_Af_{\varsigma}(y):=\bigcap_{\varsigma =1}^{\infty}\bigcap_{\delta>0} \overline{\bigcup_{\substack{k\geq \varsigma  \\ 0<|x-y|<\delta}}\partial_Af_k(y)}.
\]
We point out that this result is a generalization of Ioffe's result \cite[Theorem 3]{Ioffe84}.

We are interested in the subdifferential of $\sigma_{H:L}:H^1(M;\R)\to \R$ because its elements are rotation vectors of $\mathcal{M}(\phi_H)$-measures (Theorem \ref{thm1}). We recall that, under the canonical identification $H_1(M;\R)=H^1_{dR}(M;\R)^*$, the rotation vector $\rho(\mu)\in H_1(M;\R)$ of a $\phi_H$-invariant measure $\mu \in \mathcal{M}(\phi_H)$ is given by  
\begin{equation}
\label{eq10}
\langle [\eta],\rho(\mu)\rangle =\int \langle \eta,X_H \rangle \ d\mu, \quad [\eta]\in H^1_{dR}(M;\R).
\end{equation}
For further details and in-depth explanations of rotation vectors we refer to \cite{Schwartzman57} or \cite{Sorrentino15}. 

\section{Proofs of results}
\label{secproof}
Before proving the results from Section \ref{MatherFloer} we present the remaining proofs of results from Section \ref{secC0}. This means proving Theorem \ref{thm01}. Before doing that we prove a lemma which was implicitely used in the statement of Theorem \ref{cor00}.
\begin{lemma}
	\label{lem00}
	Let $(M,\omega)$ be weakly exact (i.e. $\omega|_{\pi_2(M)}\equiv 0$) and denote by $L\subset (M,\omega)$ a Lagrangian satisfying either the condition that $\pi_1(L)$ is Abelian or that the map $\pi_1(L)\to \pi_1(M)$ is trivial. Then there exists $\tilde{c}\in H^1(L;\R)$ such that 
	\[
	\partial \tilde{c}|_{\pi_2(M,L)}\equiv [\omega]|_{\pi_2(M,L)}.
	\]
\end{lemma}
\begin{proof}
	Since $(M,\omega)$ is weakly exact we may view $[\omega]$ as a homomorphism 
	\[
	\pi_2(M,L)/\pi_2(M)\to \R.
	\]
	By the long exact sequence for homotopy groups we may view $\pi_2(M,L)/\pi_2(M)$ as a subgroup $\pi_2(M,L)/\pi_2(M) \leq \pi_1(L)$. If $\pi_1(L)\to \pi_1(M)$ is trivial then 
	\[
	\pi_2(M,L)/\pi_2(M) \cong \pi_1(L),
	\]
	so $[\omega]$ gives rise to a homomorphism $\pi_1(L)\to \R$ which is given by an element $\tilde{c}\in H^1(L;\R)$. In case $\pi_1(L)$ is Abelian (so $\pi_1(L)=H_1(L;\Z)$) $[\omega]$ descends to a homomorphism $G\to \R$, where $G$ denotes the image of $\pi_2(M,L)/\pi_2(M)$ under the quotient map
	\[
	\pi_1(L)\to \pi_1(L)/\text{Torsion}.
	\]
	Any homomorphism $G\to \R$ extends to a homomorphism $\pi_1(L)/\text{Torsion}\to \R$ by elementary algebra (use e.g. \cite[Chapter II, Theorem 1.6]{Hungerford80}). The composition
	\[
	\pi_1(L)\to \pi_1(L)/\text{Torsion}\to \R
	\]
	extends $[\omega]$ and corresponds to an element $\tilde{c}\in H^1(L;\R)$.
\end{proof}

\subsection{Proof of Theorem \ref{thm00}}
\label{secproofthm00}
The proof relies heavily on the work of Entov-Polterovich \cite{EntovPolterovich17} as well as the work of Polterovich \cite{Polterovich14}. Given a 1-form $\eta$ on $M$ we denote by $X_{\eta}$ the vector field defined by requiring 
\[
\iota_{X_{\eta}}\omega =\eta.
\]
Given an autonomous Hamiltonian $H\in C^{\infty}(M)$, the Poisson bracket $\{H,\eta\}$ is the function given by
\[
\{H,\eta\}:=\omega(X_{\eta},X_H)=\eta(X_H)=dH(X_{\eta}).
\]
Given an open subset $U\subset M$ we consider the following quantity which is a slight modification of a quantity appearing in \cite{Polterovich14}. 
\[
pb^+_{H,U}(c):=\inf_{\eta \in c}\max_{x\in U}\{H,\eta\}(x), \quad c\in H^1(U;\R).
\]
If $X\subset U$ is a compact subset we also define 
\[
pb_{H,U}^+(c;X):=\inf_{\eta \in c}\max_{x\in X}\{H,\eta\}(x), \quad c\in H^1(U;\R).
\] 
The following theorem which is due to Polterovich is the main tool which allows us to construct invariant measures with given rotation vectors using $pb^+_{H,U}$. In \cite{Polterovich14} the statement appears in a slightly different form, but the version presented here follows directly from the theory developed in \cite{Polterovich14}.
\begin{thm}[\cite{Polterovich14}]
	\label{thm02}
	Let $H\in C^{\infty}(M)$ be an autonomous Hamiltonian which is proper and bounded from below. Given $k\in \R$ and $\delta>0$ we set $X:=\{H\leq k-\delta\}$. Suppose for some $c\in H^1(U;\R)$ there exists $\epsilon >0$ such that
	\[
	pb^+_{H,\Sigma_k}(c;X)\geq \epsilon .
	\]
	Then there exists a $\mu \in \mathcal{M}(\Sigma_k;\phi_H)$ with $\supp(\mu) \subset X$, which satisfies 
	\[
	\langle c,\rho(\mu) \rangle \geq \epsilon.
	\]
\end{thm}

The following proof is a small variation of the proof of the main technical result (Proposition 5.1) in Entov-Polterovich's \cite{EntovPolterovich17}.

\begin{proof}[Proof of Theorem \ref{thm00}]
	Recall that we consider $c\in H^1(L;\R)$ and $k\in \R$ such that 
	\[
	k\leq \kappa_{H:L}(c).
	\]
	Fix now 
	\[
	a\in \sh(\Sigma_k;L,c)\subset H^1(\Sigma_k;\R).
	\] 
	This means (by definition) that there exists a Lagrange isotopy $\psi:[0,1]\times L \to M$, starting at $L$, which satisfies
	\begin{equation}
		\label{eq02}
		L_1:=\psi_1(L)\subset \Sigma_k \quad \text{and} \quad \Flux_L(\psi)-c=\psi_1^*a \in H^1(L;\R).
	\end{equation}
	Choose $\delta>0$ so small that $\max_{L_1}(H)+\delta<k-\delta$. For every $\epsilon >0$ we can find a smooth function $u:\R \to \R$ such that 
	\[
	u(t)=\left\{
	\begin{array}{ll}
	\max_{L_1}(H)+\delta,& \text{if} \ t\leq \max_{L_1}(H)+\delta \\
	k-\delta,& \text{if} \ t\geq k-\delta,
	\end{array}
	\right.
	\]
	and $0\leq u'(t)\leq 1+\epsilon$. Consider now the Hamiltonian 
	\[
	F:=u\circ H -(k-\delta). 
	\]
	Setting $X:=\{H\leq k-\delta\}$ it suffices (by Theorem \ref{thm02}) to show that 
	\[
	pb^+_{H,\Sigma_k}(-a,X)\geq \Delta:=k-\max_{L_1}(H)-2\delta>0.
	\] 
	Since $\max_X\{F,\eta\}\leq (1+\epsilon)\max_X\{H,\eta\}$ for every 1-form $\eta$ it suffices (by letting $\epsilon \downarrow 0$) to show that $pb^+_{F,\Sigma_k}(-a,X)\geq \Delta$. To see that this inequality holds we choose a closed $1$-form $\eta$ on $\Sigma_k$ with $[\eta]=-a$. Our task is to show that $\max_{X}\{F,\eta\}  \geq \Delta$. Suppose for contradiction that this were not the case, i.e. that 
	\[
	\max_{X}\{F,\eta\}  < \Delta.
	\]
	Exactly as in \cite{EntovPolterovich17} we then have that the closed $2$-form  
	\[
	\omega_s:=\omega-sdF\wedge \eta
	\]
	is symplectic for all $s\in [0,T]$, where $T:=\frac{1}{\Delta}$. Note that $\omega_s$ is a well-defined 2-form on all of $M$ because $dF=0$ on $M\backslash X$. Hence, we can define a family of vectorfields $(Y_s)_{s\in [0,T]}$ by requiring
	\begin{equation}
		\label{eq01}
		\iota_{Y_s}\omega_s =F\eta.
	\end{equation}
	Since $F$ is compactly supported we see that $Y_s$ is compactly supported, so it integrates to an isotopy $(\rho_s)_{s\in [0,T]}$ which by Moser's argument \cite{Moser65} satisfies $\rho_s^*\omega_s=\omega$ for all $s\in [0,T]$. Since $L_1\subset (M,\omega)$ is Lagrangian and $F$ is constant on $L_1$ we have $\omega_s|_{L_1}\equiv 0$, so $\rho_s^{-1}(L_1)\subset (M,\omega)$ is also Lagrangian for all $s\in [0,T]$. Hence, we obtain a Lagrange isotopy $\tilde{\psi}:[0,1]\times L \to M$ by defining
	\[
	\tilde{\psi}_t:= \left\{
	\begin{array}{ll}
	\psi_{2t}, & \text{if} \ t\in [0,\tfrac{1}{2}] \\
	\rho_{(2t-1)T}^{-1}\circ \psi_1, & \text{if} \ t\in [\tfrac{1}{2},1],
	\end{array}
	\right.
	\]
	which starts at $L$ and is smooth after a reparametrisation. Now note that because $\supp(F)\subset X$ it follows from (\ref{eq01}) that $\supp(Y_s)\subset X$ for all $s\in [0,T]$. In particular $\rho_s|_{M\backslash X}=\id_{M\backslash X}$ for all $s\in [0,T]$. Since $L_1\subset X$ it follows that $\tilde{\psi}_1(L)=\rho_{T}^{-1}(L_1)\subset X$, meaning  
	\begin{equation}
		\label{eq04}
		\max_{\tilde{\psi}_1(L)}H\leq k-\delta.
	\end{equation}
	Since the family of vectorfields $(Z_s)_{s}$ generating $(\rho_s^{-1})_s$ is given by 
	\[
	Z_s(x)=-d\rho_s^{-1}(\rho_s(x))\cdot Y_s(\rho_s(x))
	\]
	it is easy to check that 
	\[
	\omega(Z_s,v)=-\rho_s^*(F\eta)(v) \quad \forall \ v\in T_xM.
	\]
	From $F|_{L_1}\equiv -\Delta$ it now follows that 
	\[
	\Flux_{L_1}((\rho^{-1}_{s})_{s\in [0,T]})=\int_0^T\Delta [\eta|_{L_1}]ds=-a|_{L_1}\in H^1(L_1;\R),
	\]
	and thus (using (\ref{eq02}))
	\[
	\Flux_L(\tilde{\psi})=\Flux_L(\psi)-\psi_1^*a=c.
	\]
	In particular we can estimate
	\[
	\max_{\tilde{\psi}_1(L)}H\geq \kappa_{H:L}(c)\geq k.
	\]
	But this contradicts (\ref{eq04}). This contradiction shows that $\max_{M}\{F,\eta\}  \geq \Delta$ which by Theorem \ref{thm02} finishes the proof of the existence of $\mu$. In order to see that $\mu$ can be chosen to be ergodic we consider the function 
	\begin{align*}
		R:\mathcal{M}(X;\phi_H)&\to \R,\quad \nu \mapsto R(\nu):=\langle a,\rho(\nu)\rangle.
	\end{align*}
	$\Image(R)$ is a closed interval: $\Image(R)=[A,B]$. The above proof shows that $A<0$. By \cite[Lemma 2.2.1]{BiranPolterovichSalamon03} there exists an ergodic $\nu \in \mathcal{M}(X;\phi_H)$ such that $R(\nu)=A$. Since $\mathcal{M}(X;\phi_H)\subset \mathcal{M}(\Sigma_k;\phi_H)$ this finishes the proof of Theorem \ref{thm00}. 
\end{proof}

\begin{rem}
Combining the proof of Theorem \ref{thm02} and the above proof, it is easy to deduce that the constructed measure $\mu$ satisfies $H(\supp(\mu))\subset [\max_{L_1}H+\delta,k-\delta]$.
\end{rem}

\subsection{Proofs of results from Section \ref{MatherFloer}}
\label{secproof2}

\begin{proof}[Proof of Corollary \ref{lem1}]
We first prove that $\sigma_{H:L}$ descends to a Lipschitz function 
\[
\alpha_{H:L}:H^1(L;\R) \to \R
\]
under the assumption that $r_L:H^1(M;\R) \to H^1(L;\R)$ is onto. Suppose $c,c' \in H^1(M;\R)$ satisfy $r_L(c)=r_L(c')\in H^1(L;\R)$. Fix a closed 1-form $\eta$ on $M$ with $[\eta]=c-c'$ and denote by $\psi=\{\psi_t\}_{t\in \R}$ the symplectic isotopy generated by the vector field $X_{\eta}$ characterised by
\[
\iota_{X_{\eta}}\omega=\eta.
\]
By property c) of Theorem \ref{prop1} we can now compute
\[
\sigma_{H:L}(c)=\sigma_{H:L}(c'+(c-c'))=\sigma_{H:\psi_1(L)}(c').
\]
We can view $\psi$ as a Lagrange isotopy starting at $L$
\begin{equation}
\label{eq27}
\psi|_L:[0,1]\times L \to M,
\end{equation}
with \emph{Lagrangian Flux path} $s\mapsto \Flux_L(\{ \psi_t|_L\}_{0\leq t\leq s})\in H^1(L;\R)$. We compute that 
\[
\Flux_L(\{ \psi_t|_L\}_{0\leq t\leq s})=\int_0^s r_L[\iota_{X_{\eta}}\omega] dt=r_L(c-c')s=0\in H^1(L;\R)
\]
for all $s\in [0,1]$. Hence, (\ref{eq27}) is an exact Lagrange isotopy, so there exists $\phi \in \Ham(M,\omega)$ such that $\phi(L)=\psi_1(L)$. Applying property d) of Theorem \ref{prop1} we conclude that $\sigma_{H:\psi_1(L)}(c')=\sigma_{H:\phi(L)}(c')=\sigma_{H:L}(c')$. It follows that, under the assumption that $r_L:H^1(M;\R) \to H^1(L;\R)$ is onto, $\alpha_{H:L}:H^1(L;\R) \to \R$,
\[
\alpha_{H:L}(r_L(c)):=\sigma_{H:L}(c), \quad c\in H^1(M;\R)
\]
is well-defined. Clearly $\alpha_{H:L}$ is Lipschitz. The second part of Corollary \ref{lem1} follows from Theorem \ref{thm1} if only we prove that 
\[
\partial \sigma_{H:L}(c)=i_L(\partial \alpha_{H:L}(r_L(c)))\quad \forall \ c\in H^1(M;\R),
\]
which can be seen as follows.\footnote{I am grateful to the anonymous referee whose careful advice significantly simplified this part of the proof.}
Note first that, under the canonical identification $H_1(-;\R)\cong H^1(-;\R)^*$ this equality can be written 
\[
\partial \sigma_{H:L}(c)=r_L^*(\partial \alpha_{H:L}(r_L(c)))\quad \forall \ c\in H^1(M;\R).
\]
The "$\supset$" inclusion is not hard to see: By definition, we have $h'\in \partial \alpha_{H:L}(r_L(c))$ if and only if 
\[
\forall \ c'\in H^1(L;\R): \ \langle c',h'\rangle \leq \alpha_{H:L}^{\circ}(r_L(c),c'). 
\]
In particular, if $h'\in \partial \alpha_{H:L}(r_L(c)))$ then
\[
\forall \ c'\in H^1(M;\R): \ \langle c',r_L^*(h')\rangle=\langle r_L(c'),h'\rangle \leq \alpha_{H:L}^{\circ}(r_L(c),r_L(c'))=\sigma_{H:L}^{\circ}(c,c'),
\]
where in the last equality we use that $r_L$ is linear. This shows 
\[
\partial \sigma_{H:L}(c)\supset r_L^*(\partial \alpha_{H:L}(r_L(c))).
\]
For the "$\subset$" inclusion, note that 
\[
h\in \partial \sigma_{H:L}(c)\Rightarrow \left( \forall \ c'\in \ker(r_L): \langle c',h \rangle \leq \sigma^{\circ}_{H:L}(c,c')= \alpha_{H:L}^{\circ}(r_L(c),0)=0 \right).
\]
I.e. $\partial \sigma_{H:L}(c)\subset \ker(r_L)^{\perp}$, where the latter denotes the annihilator of $\ker(r_L)$. Since the image of $r^*_L$ coincides with the annihilator of $\ker(r_L)$ (because $H^1(M;\R)$ and $H^1(L;\R)$ are finite dimensional), this is equivalent to $\partial \sigma_{H:L}(c)\subset \Image(r_L^*)$. I.e. if $h\in \partial \sigma_{H:L}(c)$ then $h=r_L^*(h')$ for some $h'\in H^1(L;\R)^*$ and 
\[
\forall \ c'\in H^1(M;\R): \langle r_L(c'),h' \rangle = \langle c',h \rangle \leq \sigma_{H:L}^{\circ}(c,c')=\alpha_{H:L}^{\circ}(r_L(c),r_L(c')).
\]
Since $r_L$ is assumed surjective we conclude that $h'\in \partial \alpha_{H:L}(r_L(c))$, which shows 
\[
\partial \sigma_{H:L}(c)\subset r_L^*(\partial \alpha_{H:L}(r_L(c))).
\]
\end{proof}

\begin{proof}[Proof of Corollary \ref{cor2}]
Consider the path $\gamma:[0,1]\to H^1(L;\R)$ given by $\gamma(t)=tc$. By Lebourg's mean value theorem (see Section \ref{secmeas}) there exists $s\in (0,1)$ and $h\in \partial \alpha_{H:L}(\gamma(s))$ such that 
\[
\langle c,h\rangle=\alpha_{H:L}(c)-\alpha_{H:L}(0).
\]
The existence of $\mu \in \mathcal{M}(\phi_H)$ with $\rho(\mu)=i_L(h)$ now follows from Corollary \ref{lem1}. To get the estimate note that, given a Lagrange isotopy $\psi:[0,1]\times L\to M$ with $\psi_0$ the inclusion and $\Flux_{L}(\psi)=c$ there exists (by a result due to Solomon \cite[Lemma 6.6]{Solomon13}) a symplectic isotopy $\tilde{\psi}:[0,1]\times M \to M$ such that $\tilde{\psi}_t(L)=\psi_t(L)$ and such that $r_L(\Flux(\tilde{\psi}))=c$. Using property b) of Theorem \ref{prop1} we can therefore estimate 
\[
\langle c,h\rangle \geq \min_{x\in \tilde{\psi}_1(L)}H(x)-\max_{x\in L}H(x)=\min_{x\in \psi_1(L)}H(x)-\max_{x\in L}H(x).
\]
\end{proof}

\begin{proof}[Proof of Lemma \ref{lemmat1}]
Recall that we need to show $\alpha_{H_{\epsilon},T(0,0)}(I,0)=0$ for $I\in \R$, where 
\[
H_{\epsilon}(\theta,I)=H_0(I)+\epsilon F(\theta,I),
\]
for 
\[
H_0(I)=I_1I_2 \quad \& \quad F(\theta,I)=\varphi(I_1)\sin(2\pi \theta_1).
\]
By Theorem \ref{prop1} we have 
\[
\alpha_{H_{\epsilon},T(0,0)}(I_1,I_2)=\alpha_{H_{\epsilon},T(I_1,I_2)}(0,0) \quad \forall \ (I_1,I_2)\in \R^2.
\]
Now fix $I\in \R$. We need to show that $\alpha_{H_{\epsilon},T(I,0)}(0,0)=0$. We will think of $\T^2\times \R^2=(S^1\times \R)\times (S^1\times \R)$ with coordinates $(\theta_1,I_1)$ on the first factor and $(\theta_2,I_2)$ on the second. Choose an embedded circle $S\subset (S^1\times \R)$ such that $S$ is Hamiltonian isotopic to $\{I_1=I\}$ and 
\[
S\cap \{I_1\leq K+2\}=\{(\theta_1,I_1)\ |\ \theta_1\in \{ 0,\tfrac{1}{2}\}\  \& \ |I_1|\leq K+2 \}.
\]
Then there exists $\phi \in \Ham(\T^2 \times \R^2,d\lambda)$ such that $\phi(T(I,0))=S\times \{I_2=0\}\subset (S^1\times \R)\times (S^1\times \R)$ and applying Theorem \ref{prop1} gives
\[
\alpha_{H_{\epsilon},T(I,0)}(0,0)=\alpha_{H_{\epsilon},S\times \{I_2=0\}}(0,0).
\]
Since $H_{\epsilon}|_{S\times \{I_2=0\}}\equiv 0$ for all $\epsilon \geq 0$ it follows from point b) of Theorem \ref{prop1} that $\alpha_{H_{\epsilon},S\times \{I_2=0\}}(0,0)=0$ for all $\epsilon \geq 0$. This finishes the proof.
\end{proof}

The next several pages take up the proof of Theorem \ref{thm1}. First let's consider the setup: Set $m=\dim_{\R}H^1(M;\R)$ and fix once and for all a basis $e_1,\ldots ,e_m$ for $H^1(M;\R)$. Choose closed one forms $\eta_1,\ldots ,\eta_m$ on $M$ such that $e_l=[\eta_l]$ and define vector fields $X_{\eta_1},\ldots X_{\eta_m}$ by requiring
\[
\iota_{X_{\eta_l}}\omega=\eta_l \quad \forall \ l=1,\ldots ,m.
\] 
Denote the symplectic isotopy generated by $X_{\eta_l}$ by $\psi^l=\{ \psi_t^l \}_{t\in \R}$. For each $k\in \N$ we define a function $a_k:H^1(M;\R)\to \R$ by
\[
a_k(c)=\frac{l_L(\psi^m_{-\kappa_m} \cdots \psi^1_{-\kappa_1} \tilde{\phi}_H^k\psi^1_{\kappa_1} \cdots \psi^m_{\kappa_m})}{k}
\]
if $c=\sum_{l=1}^{m}\kappa_le_l$. The following lemma is due to Vichery \cite{Vichery14}. In \cite{Vichery14} the case of the zero-section in a cotangent bundle is considered. But in the general case the same proof applies.  

\begin{lemma}[\cite{Vichery14}]
\label{lemVichery}
The sequence $(a_k)_{k\in \N}$ converges uniformly to $\sigma_{H:L}$ on compact subsets.
\end{lemma}

\begin{prop}
\label{prop2}
Let $h\in H_1(M;\R)$ be a Dini subdifferential of $a_k$ at the point $c=\sum_{l=1}^{m}\kappa_l e_l\in H^1(M;\R)$.\footnote{See (\ref{eq1}).} We associate to any $c'=\sum_{l=1}^{m}\kappa'_l e_l\in H^1(M;\R)$ the symplectic isotopy 
\begin{equation}
\label{eq11}
\psi_{\tau}:=\psi^1_{\kappa_1+\tau \kappa'_1}  \cdots  \psi^m_{\kappa_m+\tau \kappa'_m}
\end{equation}
and denote by $X_{\tau}(\psi_{\tau}(x)):=\left. \frac{d}{ds}\right|_{s=0}\psi_{s+\tau}(x)$ its infinitesimal generator. Then there exists a point $x\in \psi_0(L)$ such that $\phi_H^k(x)\in \psi_0(L)$ and  
\begin{equation}
\label{eq12}
\langle c',h \rangle \leq \frac{1}{k}\int_{0}^{k}\langle dH,X_0\rangle \phi_H^t(x)\ dt.
\end{equation}
Moreover, the curve $\gamma:([0,k],\{0,k\})\to (M,L)$ given by $\gamma(t)=\phi_H^t(x)$ represents the trivial element in $\pi_1(M,\psi_0(L))$ and there exists a capping $\widehat{\gamma}$ of $\gamma$ such that
\begin{align}
\label{eq5}
a_k(c)&=\frac{1}{k}\int_0^kH(\gamma(t))\ dt-\frac{1}{k}\int \widehat{\gamma}^*\omega.
\end{align} 

\end{prop}
\begin{proof}
Let us first prove the result under the non-degeneracy assumption
\begin{equation}
\label{eq13}
\phi_{H\psi_0}^k(L)\pitchfork L.
\end{equation}
The Hamiltonian $kH$ generates the isotopy $t\mapsto \phi_H^{tk}$. Choose a generic path $J=\{J_t\}_{t\in [0,1]}$ of $\omega$-compatible almost complex structures so that we have a well-defined Floer chain complex $(CF(L{:}kH\psi_0,J),d)$. Choose a function $\chi \in C^{\infty}(\R;[0,1])$ such that 
\[
\chi(s)=
\left\{
\begin{array}{ll}
0, & \text{if} \ s\leq 0 \\
1, & \text{if} \ s\geq 1 
\end{array}
\right.
\]
and $\chi' \geq 0$. Now fix $\tau >0$ so small that $\phi^k_{H\psi_{\tau}}(L)\pitchfork L$ and consider the $s$-dependent Hamiltonian\footnote{Here and in the rest of the proof one should \emph{not} think of $\tau$ as a continuous variable, but as a fixed parameter.}
\begin{equation}
\label{eq7}
F^{\tau}_s(x)=kH\psi_0(x)+\chi(s)(kH\psi_{\tau}(x)-kH\psi_0(x)).
\end{equation}
There exists a regular homotopy of Floer data $(K_{\tau, s},J^{\tau, s})_{s\in \R}$ which is stationary for $s\notin (0,1)$ such that 
\begin{align}
\label{eq6}
|F^{\tau}-K^{\tau}|_{C^{\infty}(M\times [0,1] \times [0,1])}&\leq \tau^2 \\
|J-J^{\tau, s}|_{C^{\infty}(M\times [0,1])}&\leq \tau^2 \quad \forall \ s\in (0,1), \nonumber
\end{align}
The Floer continuation map 
\[
\Phi: CF_*(L{:}kH\psi_0,J)\to CF_*(L{:}kH\psi_{\tau},J^{\tau,1})
\] 
is defined "by counting" finite-energy solutions $u\in C^{\infty}(\R \times [0,1];M)$ to the problem
\begin{equation}
\label{eq3}
\left\{
\begin{array}{ll}
\partial_su+J_t^{\tau, s}(u)(\partial_tu-X_{K_s}(u))=0 \\
u(\R \times \{0,1\})\subset L,
\end{array}
\right.
\end{equation}
where the \emph{energy of $u$} denotes the non-negative quantity 
\[
E_{J^{\tau}}(u)=\int_{-\infty}^{\infty}\int_0^1 \omega(\partial_su,J^{\tau,s}(u)\partial_su)\ dtds.
\]

By a very general result due to Usher \cite[Theorem 1.4]{Usher08}, there exists a cycle 
\[
\mathfrak{c}\in CF_*(L{:}kH\psi_0,J),
\]
all of whose elements have action $\leq l_L(\psi_0^{-1}\tilde{\phi}_H^k\psi_0)$ and which satisfies $[\mathfrak{c}]=\PSS([L])\in HF_*(L{:}kH\psi_0,J)$.\footnote{See Section \ref{seqprem} for the notation used here.} Since we also have $\Phi(\mathfrak{c})=\PSS([L])$, at least one of the Hamiltonian chords of the cycle $\Phi(\mathfrak{c})$ must have action $\geq l_L(\psi_{\tau}^{-1}\tilde{\phi}_H^k\psi_{\tau})$. We hence deduce the existence of a finite-energy solution $u^{\tau}_s(t)=u^{\tau}(s,t)$ to (\ref{eq3}) such that
\begin{equation}
\label{eq4}
\mathcal{A}_{kH\psi_0:L}(u^{\tau}_{-\infty}) \leq  l_L(\psi_{0}^{-1}\tilde{\phi}_H^k\psi_{0}) \quad \text{and} \quad l_L(\psi_{\tau}^{-1}\tilde{\phi}_H^k\psi_{\tau}) \leq \mathcal{A}_{kH\psi_{\tau}:L}(u^{\tau}_{\infty}).
\end{equation}
Here $u_{-\infty}^{\tau}$ (respectively $u_{\infty}^{\tau}$) denotes one of the elements of the chain $\mathfrak{c}$ (respectively $\Phi(\mathfrak{c})$). Integrating by parts reveals that $u^{\tau}$ satisfies the energy identity 
\begin{align}
\label{eq30}
0\leq E_{J^{\tau}}(u^{\tau})=\mathcal{A}_{kH\psi_0:L}(u^{\tau}_{-\infty})-\mathcal{A}_{kH\psi_{\tau}:L}(u^{\tau}_{\infty})+\int_{-\infty}^{\infty}\int_0^1 \partial_sK_s(u^{\tau}(s,t))dtds.
\end{align}
Using (\ref{eq6}) and (\ref{eq4}), this identity implies the estimate
\begin{align}
\label{eq9}
\frac{a_k(c+\tau c')-a_k(c)}{\tau}&= \frac{l_L(\psi_{\tau}^{-1}\tilde{\phi}_H^k\psi_{\tau})-l_L(\psi_{0}^{-1}\tilde{\phi}_H^k\psi_{0})}{k\tau} \nonumber \\
&\leq \frac{\mathcal{A}_{kH\psi_{\tau}:L}(u^{\tau}_{\infty})-\mathcal{A}_{kH\psi_0:L}(u^{\tau}_{-\infty})}{k\tau} \nonumber \\
&\leq \frac{1}{k\tau}\int_{-\infty}^{\infty}\int_0^1 \partial_sK_s(u^{\tau}(s,t))dtds  \\ 
&\leq  \frac{1}{k\tau}\int_{-\infty}^{\infty}\int_0^1 \partial_sF^{\tau}_s(u^{\tau}(s,t))dtds +\frac{\tau}{k} \nonumber \\
&= \int_{-\infty}^{\infty}\int_0^1 \chi'(s)\left( \frac{H\psi_{\tau}(u^{\tau}(s,t))-H\psi_{0}(u^{\tau}(s,t))}{\tau} \right) dtds +\frac{\tau}{k}. \nonumber
\end{align}
Moreover, we can use the energy identity together with (\ref{eq4}) to estimate the energy of $u^{\tau}$ from above: 
\begin{align}
\label{eq15}
E_{J^{\tau}}(u^{\tau})&\leq  l_L(\psi_{0}^{-1}\tilde{\phi}_H^k\psi_{0})-l_L(\psi_{\tau}^{-1}\tilde{\phi}_H^k\psi_{\tau})+k\max_{M}|H\psi_{\tau}-H\psi_0|+\tau^2 \\
&\leq 2k\max_{M}|H\psi_{\tau}-H\psi_0|+\tau^2. \nonumber 
\end{align}
The last inequality makes use of the Lipschitz property of Lagrangian spectral invariants (see \cite{LeclercqZapolsky15}).
Choose now a sequence $(\tau_{\varsigma})_{\varsigma \in \N}\subset (0,1)$ such that $\tau_{\varsigma} \downarrow 0$ and 
\begin{equation}
\label{eq17}
\frac{a_k(c+\tau_{\varsigma} c')-a_k(c)}{\tau_{\varsigma}}\stackrel{\varsigma \to \infty}{\longrightarrow} \liminf_{\tau \downarrow 0}\frac{a_k(c+\tau c')-a_k(c)}{\tau}
\end{equation}
For each $\tau_{\varsigma}$ we find a Floer trajectory $u^{\varsigma}:=u^{\tau_{\varsigma}}$ meeting the requirement (\ref{eq4}) with $\tau=\tau_{\varsigma}$. Applying Gromov convergence to the sequence $(u^{\varsigma})_{\varsigma \in \N}$ we obtain a subsequence (still denoted by $(u^{\varsigma})_{\varsigma \in \N}$) which $C^{\infty}_{loc}$-converges. This follows from a result due to Hofer \cite[Proof of Proposition 2]{Hofer88} saying that, in our setup, the obstruction to $C^{\infty}_{loc}$-precompactness is "bubbling off" of a non-constant holomorphic disc or sphere. However, this cannot occur to $(u^{\varsigma})_{\varsigma \in \N}$ since the estimate (\ref{eq15}) implies 
\[
E_{J^{\tau_{\varsigma}}}(u^{\varsigma})\stackrel{\varsigma \to \infty}{\longrightarrow} 0.
\]
By Fatou's lemma, this also implies that the pointwise limit $u$ of $(u^{\varsigma})_{\varsigma \in \N}$ is an $s$-independent map given by $u(s,t)=u(t)=\phi_{kH\psi_0}^t(y)$ for some $y\in L$. 
Now, using $\chi' \geq 0$ and the mean value theorem, we find for every $(s,t)\in \R \times [0,1]$ a $\tau'_{\varsigma}(s,t)\in (0,\tau_{\varsigma})$ such that
\begin{align*}
\left| \chi'(s)\left( \frac{H\psi_{\tau_{\varsigma}}(u^{\varsigma}(s,t))-H\psi_{0}(u^{\varsigma}(s,t))}{\tau_{\varsigma}} \right) \right|&= \chi'(s)|\langle dH,X_{\tau_{\varsigma}'(s,t)} \rangle|(\psi_{\tau_{\varsigma}'(s,t)}(u^{\varsigma}(s,t))) \\
&\leq \chi'(s)\max_{(\tau,z)\in [0,1]\times M}|\langle dH,X_{\tau} \rangle|(z)
\end{align*}
Since $h$ is a Dini subdifferential of $a_k$ at $c$ and 
\[
\int_{-\infty}^{\infty} \int_{0}^{1}\chi'(s)\max_{(\tau,z)\in [0,1]\times M}|\langle dH,X_{\tau} \rangle|(z) \ dtds=\max_{(\tau,z)\in [0,1]\times M}|\langle dH,X_{\tau} \rangle|(z)<\infty,
\]
we can apply Lebesgue dominated convergence in the estimates (\ref{eq9}) to conclude that 
\begin{align*}
\langle c', h \rangle &\leq \lim_{\varsigma \to \infty} \int_{-\infty}^{\infty}\int_0^1 \chi'(s)\left( \frac{H\psi_{\tau_{\varsigma}}(u^{\varsigma}(s,t))-H\psi_{0}(u^{\varsigma}(s,t))}{\tau_{\varsigma}} \right) dtds \\
&= \int_{-\infty}^{\infty}\int_0^1 \chi'(s)\lim_{\varsigma \to \infty}\left( \frac{H\psi_{\tau_{\varsigma}}(u^{\varsigma}(s,t))-H\psi_{0}(u^{\varsigma}(s,t))}{\tau_{\varsigma}} \right) dtds \\
&=\int_{-\infty}^{\infty}\int_0^1 \chi'(s) \langle dH,X_0 \rangle  \psi_0(u(t)) dtds =\int_0^1 \langle dH,X_0 \rangle  \phi_{H}^{kt}\psi_0(y)dt,
\end{align*}
where we use the fact that $\psi_0\phi_{kH\psi_0}^{t}(y)=\phi_{kH}^{t}\psi_0(y)=\phi_{H}^{kt}\psi_0(y)$. By changing variables to $s=kt$ one now concludes that (\ref{eq12}) holds for $x:=\psi_0(y)\in \psi_0(L)$. In order to find a capping $\widehat{\gamma}$ for $\gamma(t)=\phi_H^t(x),\ t\in [0,k]$ such that (\ref{eq5}) holds, we first note that
\begin{equation}
\label{eq16}
\mathcal{A}_{kH\psi_0:L}(u^{\varsigma}_{-\infty}),\mathcal{A}_{kH\psi_{\tau_{\varsigma}}:L}(u^{\varsigma}_{\infty}) \stackrel{\varsigma \to \infty}{\longrightarrow} l_L(\psi_{0}^{-1}\tilde{\phi}_H^k\psi_{0}),
\end{equation}
which is easily seen using the estimates (\ref{eq4}) and (\ref{eq30}). By $C^{\infty}_{loc}$-convergence we may choose $\varsigma$ so large that the curve $[0,1]\ni t \mapsto u^{\varsigma}_s(t)=u^{\varsigma}(s,t)$ is $C^0$-close to $u$ for every $s\in [-2,2]$. In particular we can choose $\varsigma$ so large that the areas of the two cappings of $u$, obtained by adjusting the cappings of $u^{\varsigma}_{-2}$ and $u^{\varsigma}_{2}$ given by $u^{\varsigma}$ slightly, differ by less than $\tau_L$. Hence, by monotonicity of $L$ they have the same area and thus define a class of cappings $\widetilde{u}$ of $u$. Given $\epsilon >0$ we can in addition achieve 
\[
|\mathcal{A}_{kH\psi_0}(u^{\varsigma}_{-2})-\mathcal{A}_{kH\psi_0}(\widetilde{u})|<\epsilon \quad \text{and}\quad |\mathcal{A}_{kH\psi_{\tau_{\varsigma}}}(u^{\varsigma}_{2})-\mathcal{A}_{kH\psi_{\tau_{\varsigma}}}(\widetilde{u})|<\epsilon.
\]
Since $\mathcal{A}_{kH\psi_{\tau_{\varsigma}}}(\widetilde{u}) \to \mathcal{A}_{kH\psi_{0}}(\widetilde{u})$ for $\varsigma \to \infty$ these estimates together with
\begin{align*}
0&\leq \mathcal{A}_{kH\psi_{\tau_{\varsigma}}:L}(u^{\varsigma}_2)-\mathcal{A}_{kH\psi_{\tau_{\varsigma}}:L}  (u^{\varsigma}_{\infty})\leq E_{J^{\tau_{\varsigma}}}(u^{\varsigma}) \stackrel{\varsigma \to \infty}{\longrightarrow} 0 \\
0&\leq \mathcal{A}_{kH\psi_{\tau_0}:L}(u^{\varsigma}_{-\infty})-\mathcal{A}_{kH\psi_{0}:L}  (u^{\varsigma}_{-2})\leq E_{J^{\tau_{\varsigma}}}(u^{\varsigma}) \stackrel{\varsigma \to \infty}{\longrightarrow} 0
\end{align*}
and (\ref{eq16}) imply that 
\[
\mathcal{A}_{kH\psi_{0}}(\widetilde{u})=l_L(\psi_{0}^{-1}\tilde{\phi}_H^k\psi_{0}).
\]
Hence, the cappings of $\widetilde{u}$ give rise to a capping $\widehat{\gamma}$ of $\gamma$ such that (\ref{eq5}) holds.

Let's now discuss how to deal with the case when assumption (\ref{eq13}) is violated. Then we proceed as follows: Again we choose a sequence $(\tau_{\varsigma})_{\varsigma \in \N}$ such that $\tau_{\varsigma} \downarrow 0$ and (\ref{eq17}) is satisfied as well as a path $J=\{ J_{t} \}_{t \in [0,1]}$ of $\omega$-compatible almost complex structures. For each $\varsigma$ we choose $H^{\varsigma}\in C^{\infty}([0,1]\times M)$ satisfying
\begin{equation}
\label{eq18}
|H^{\varsigma}-kH|_{C^{\infty}([0,1]\times M)}\leq \tau_{\varsigma}^2
\end{equation}
as well as
\[
\phi_{H^{\varsigma}\psi_0}^1(L)\pitchfork L \quad \text{and} \quad \phi_{H^{\varsigma}\psi_{\tau_{\varsigma}}}^1(L)\pitchfork L.
\]
Now choose paths of compatible almost complex structures $J^{\varsigma}=\{ J^{\varsigma}_t \}_{\tau \in [0,1]}$ which are regular in the sense that we have well-defined Floer chain complexes 
\[
(CF(L{:}H^{\varsigma}\psi_0, J^{\varsigma}),d), \quad  (CF(L{:}H^{\varsigma}\psi_{\tau_{\varsigma}}, J^{\varsigma}),d)
\]
and such that $J^{\varsigma}$ satisfies
\[
|J^{\varsigma}-J|_{C^{\infty}([0,1]\times M)}\leq \tau_{\varsigma}^2.
\]
For each $\varsigma$ we obtain by the above procedure a Floer trajectory $u^{\varsigma}:=u^{\tau_{\varsigma}}$ such that (\ref{eq4}) is satisfied with $H=H^{\varsigma}$ and $\tau=\tau_{\varsigma}$. In particular (\ref{eq15}) follows with $H=H^{\varsigma}$ and $\tau=\tau_{\varsigma}$. Using (\ref{eq9}) and (\ref{eq18}) as well as the Lipschitz property of Lagrangian spectral invariants \cite{LeclercqZapolsky15}, we can estimate 
\begin{align}
\frac{a_k(c+\tau_{\varsigma} c')-a_k(c)}{\tau_{\varsigma}}&=\frac{l_L(\psi_{\tau}^{-1}\tilde{\phi}_H^k\psi_{\tau})-l_L(\psi_{0}^{-1}\tilde{\phi}_H^k\psi_{0})}{k\tau} \nonumber \\ 
&\leq \frac{l_L(\psi_{\tau}^{-1}\tilde{\phi}_{H^{\varsigma}}^1\psi_{\tau})-l_L(\psi_{0}^{-1}\tilde{\phi}_{H^{\varsigma}}^1\psi_{0})}{k\tau_{\varsigma}}+\frac{2\tau_{\varsigma}}{k} \label{eq19} \\
&\leq \frac{1}{k}\int_{-\infty}^{\infty}\int_0^1 \chi'(s)\left( \frac{H_t^{\varsigma}\psi_{\tau_{\varsigma}}(u^{\varsigma}(s,t))-H_t^{\varsigma}\psi_{0}(u^{\varsigma}(s,t))}{\tau_{\varsigma}} \right) dtds +\frac{3\tau_{\varsigma}}{k} \nonumber  \\
&\leq \int_{-\infty}^{\infty}\int_0^k \chi'(s)\left( \frac{H\psi_{\tau_{\varsigma}}(u^{\varsigma}(s,t))-H\psi_{0}(u^{\varsigma}(s,t))}{\tau_{\varsigma}} \right) dtds +\frac{5\tau_{\varsigma}}{k}. \nonumber
\end{align}
Again, by applying Hofer's compactness argument \cite[Proof of Proposition 2]{Hofer88}, we may assume that $(u_{\varsigma})_{\varsigma \in \N}$ $C^{\infty}_{loc}$-converges. Exactly as above the pointwise limit $u$ is a $s$-constant map given by $u(t)=\phi^{kt}_{H\psi_0}(y)$ for some $y\in L$. Via dominated convergence the estimate (\ref{eq19}) now gives
\[
\langle c', h \rangle \leq \frac{1}{k}\int_{0}^{k}\langle dH,X_0 \rangle \phi_H^t(x)\ dt
\]
for $x=\psi_0(y)\in \psi_0(L)$ and the existence of a capping is obtained exactly as before. This finishes the proof.
\end{proof}
\begin{rem}
	\label{Mrem1}
	Suppose now that $\omega=d\lambda$ is exact with $\lambda|_L=df$ for some $f\in C^{\infty}(L)$. In this case we need to relate the action appearing in (\ref{eq5}) to the one appearing in (\ref{eq8}). Choose a closed one form $\eta$ in the class $c\in H^1(M;\R)$. An easy computation shows that $[\lambda|_{\psi_0(L)}]=[\eta|_{\psi_0(L)}]\in H^1(\psi_0(L);\R)$, so in fact $\eta$ can be chosen to satisfy $(\lambda-\eta)|_{\psi_0(L)}\equiv 0$. In particular we then have
	\[
	\int \widehat{\gamma}^*\omega=\int d\widehat{\gamma}^*(\lambda-\eta)=\int_0^k \langle \lambda, X_H \rangle \gamma(t)dt - \int_0^k \langle \eta, X_H \rangle \gamma(t)dt,
	\]
	so (\ref{eq5}) reads
	\begin{equation}
	\label{Mateq901}
	a_k(c)=\mathcal{A}_{H,\lambda}(\mu_{\gamma})+\int \langle \eta, X_H \rangle d\mu_{\gamma},
	\end{equation}
	where $\mu_{\gamma}$ is the Borel probability measure obtained by pushing forward the normalized Lebesgue measure on $[0,k]$ to $M$ via $\gamma$.
\end{rem}
Though we are mainly interested in the Clarke subdifferentials of $\alpha_{H:L}$, the proof of Theorem \ref{thm1} goes via the so-called \emph{approximate subdifferential} due to Ioffe \cite{Ioffe84}.
\begin{prop}
	\label{prop3}
	Let $h\in \partial_A a_k(c)\subset H_1(M;\R)$ be an approximate subdifferential of $a_k$ at $c=\sum_{l=1}^{m}\kappa_l e_l\in H^1(M;\R)$.\footnote{See (\ref{eq222}).} Given a fixed $c'=\sum_{l=1}^{m}\kappa'_l e_l\in H^1(M;\R)$ we consider the symplectic isotopy (\ref{eq11}) and denote its infinitesimal generator by $X_{\tau}$. Then there exists a point $x\in L$ such that the curve $[0,k]\ni t\mapsto \gamma(t)=\phi_{H\psi_0}^t(x)$ satisfies (\ref{eq12}) and (\ref{eq5}).
\end{prop}
\begin{proof}
By (\ref{eq222}) there exists a sequence $(c_{\varsigma})_{\varsigma \in \N}\subset H^1(M;\R)$ such that $h$ is a Dini subdifferential of $a_k$ at $c_{\varsigma}$ and
\begin{align}
\label{eq14}
c_{\varsigma}\stackrel{\varsigma \to \infty}{\longrightarrow}c
\end{align}
Writing $c_{\varsigma}=\sum_{l=1}^m \xi_l^{\varsigma} e_l$ we define the symplectic isotopy
\[
\upsilon^{\varsigma}_{\tau}:=\psi^1_{\xi_1^{\varsigma}+\tau \kappa_1'}\cdots \psi^m_{\xi_m^{\varsigma}+\tau \kappa_m'}
\]
and denote by $X^{\varsigma}_{\tau}$ the infinitesimal generator of $\tau \mapsto \upsilon^{\varsigma}_{\tau}$. By Proposition \ref{prop2} there exists a sequence of points $(x_{\varsigma})_{\varsigma \in \N}$ with $x_{\varsigma}\in \upsilon_{0}^{\varsigma}(L)$ such that the curves $[0,k]\ni t\mapsto \gamma_{\varsigma}(t):=\phi_{H}^t(x_{\varsigma})$ have cappings $\widehat{\gamma}_{\varsigma}$ for which
\begin{align*}
\langle c',h \rangle &\leq \frac{1}{k}\int_{0}^{k}\langle dH,X^{\varsigma}_0\rangle (\gamma_{\varsigma}(t))dt \\
a_k(c_{\varsigma})&=\frac{1}{k}\int_0^kH(\gamma_{\varsigma}(t))dt-\frac{1}{k}\int \widehat{\gamma}_{\varsigma}^*\omega.
\end{align*}
By compactness of $M$, we may assume that $x_{\varsigma}\to x$ for some $x\in M$, after passing to a subsequence. The claim now follows by passing to the limit, using the fact that $a_k$ is locally Lipschitz with respect to any norm on $H^1(M;\R)$.
\end{proof}

\begin{proof}[Proof of Theorem \ref{thm1}]
Consider the set 
\[
K:=\{ \rho(\mu)\in H_1(M;\R) \ |\ \mu \in \mathcal{M}(\phi_H)\}
\]
of rotation vectors realized by $\mathcal{M}(\phi_H)$-measures. Clearly $K\subset H_1(M;\R)$ is non-empty, convex and compact. Hence, $K$ is characterized by its \emph{support function} \cite[Section 1.7.1]{Schneider14} $\chi_K:H^1(M;\R) \to \R$, defined by
\[
\chi_K(c'):=\max_{h\in K}\langle c',h \rangle, \quad c'\in H^1(M;\R).
\]
The characterization is that $h\in K$ if and only if
\begin{equation}
\label{eq25}
\langle c',h \rangle \leq \chi_K(c') \quad \forall \ c'\in H^1(M;\R).
\end{equation}
We need to show that $\partial \sigma_{H:L}(c)\subset K$ for all $c\in H^1(M;\R)$. In fact, by property c) of Theorem \ref{prop1}, it suffices to show that $\partial \sigma_{H:L}(0)\subset K$. Moreover, since $\sigma_{H:L}$ is locally Lipschitz and $K$ is convex and closed, it suffices (by (\ref{eq111})) to show that
\begin{equation}
\label{eq24}
\partial_A \sigma_{H:L}(0) \subset K.
\end{equation}
Fix some $h\in \partial_A \sigma_{H:L}(0)$. By the above discussion we need to show that (\ref{eq25}) holds, so choose some $c'=\sum_{l=1}^m \kappa_l' e_l  \in H^1(M;\R)$.
By Lemma \ref{lemVichery} $a_k \to \sigma_{H:L}$ uniformly on compact subsets. Hence, applying (\ref{eq333}) we can find an increasing sequence $(k_{\varsigma})_{\varsigma \in \N} \subset \N$ as well as sequences $(c_{\varsigma})_{\varsigma \in \N}\subset H^1(M;\R)$ and $(h_{\varsigma})_{\varsigma \in \N}\subset H_1(M;\R)$ such that $h_{\varsigma}\in \partial_A a_{k_{\varsigma}}(c_{\varsigma})$ and
\begin{equation*}
c_{\varsigma} \stackrel{\varsigma \to \infty}{\longrightarrow} 0 \quad \text{and} \quad h_{\varsigma} \stackrel{\varsigma \to \infty}{\longrightarrow} h.
\end{equation*}
Writing $c_{\varsigma}=\sum_{l=1}^m\kappa_l^{\varsigma}e_l$ we define a symplectomorphism $\psi_{\varsigma}$ by 
\[
\psi_{\varsigma}:=\psi^1_{\kappa_1^{\varsigma}} \cdots \psi^m_{\kappa_m^{\varsigma}}
\]
and denote by $X^{\varsigma}_{\tau}$ the infinitesimal generator of the symplectic isotopy
\[
\tau \mapsto \psi^1_{\kappa_1^{\varsigma}+\tau \kappa_1'} \cdots \psi^m_{\kappa_m^{\varsigma} +\tau \kappa_m'}
\]
Applying Proposition \ref{prop3} we obtain for each $\varsigma$ the existence of $x_{\varsigma}\in \psi_{\varsigma}(L)$ such that 
\[
\langle c', h_{\varsigma} \rangle \leq \frac{1}{k_{\varsigma}} \int_0^{k_{\varsigma}}\langle dH,X_{0}^{\varsigma}\rangle \phi_H^t(x_{\varsigma})\ dt.
\]
Denote by\footnote{Recall that $\mathcal{M}$ denotes the space of compactly supported Borel probability measures on $M$.} $\mu_{\varsigma}\in \mathcal{M}$ the unique probability measure characterized by 
\[
\int f\ d\mu_{\varsigma}=\frac{1}{k_{\varsigma}} \int_0^{k_{\varsigma}}f\phi_H^t(x_{\varsigma})\ dt \quad \forall \ f\in C^0(M),
\]
so that 
\begin{equation}
\label{eq26}
\langle c', h_{\varsigma} \rangle \leq \int \langle dH,X_{0}^{\varsigma}\rangle \ d\mu_{\varsigma}.
\end{equation}
By weak$^*$-compactness of $\mathcal{M}$ we may after passing to a subsequence (still denoted by $(\mu_{\varsigma})_{\varsigma \in \N}$) assume that 
\[
\mu_{\varsigma}\stackrel{w^*}{\rightharpoonup} \mu \in \mathcal{M}.
\]
By the classical Krylloff-Bogoliouboff argument \cite{KryloffBogoliouboff37} $\mu$ is $\phi_H$-invariant (see also \cite[Proposition 3.1.1]{Sorrentino15}). Passing to the limit in (\ref{eq26}) we conclude that 
\begin{equation}
\label{eq31}
\langle c',h \rangle \leq \int \langle dH,Y\rangle \ d\mu,
\end{equation}
where 
\[
Y=\sum_{l=1}^m \kappa_l'X_{\eta_l}.
\]
In particular (using $\langle dH,X_{\eta_l} \rangle =\langle \eta_l,X_H \rangle$) we obtain 
\begin{equation*}
\langle  c',h \rangle \leq \sum_{l=1}^m \kappa_l'  \int  \langle \eta_l,X_H \rangle d\mu=\sum_{l=1}^m \kappa_l' \langle e_l,\rho(\mu) \rangle=\langle c',\rho(\mu) \rangle \leq \chi_K(c').  
\end{equation*}
Since $c'\in H^1(M;\R)$ was arbitrary we conclude that (\ref{eq25}) holds, so $h\in K$ and the proof is done.
\end{proof}

We will now discuss the adaptions required to prove the results in the non-compact setting from Section \ref{noncompact}.
\begin{proof}[Sketch of proof of Theorem \ref{thm2} and Corollary \ref{cor3}]
	Recall that in the non-compact setting we consider a Liouville manifold $(M,\omega=d\lambda)$ and assume that $\lambda|_L=df$ for some $f\in C^{\infty}(M)$. One then defines each $a_k$ exactly as above, except that $\eta_1,\ldots , \eta_m$ are closed 1-forms which represent a basis for $H^1_{dR}(M;\R)$ and satisfy condition (\ref{Meq1}). Now Proposition \ref{prop2} and \ref{prop3} hold exactly as in the compact setting. Fix some $c=\sum_{l=1}^m\kappa_l [\eta_l]\in H^1_{dR}(M;\R)$ and denote by $C\subset M$ the closure of the smallest $\phi_H$-invariant set containing the subset
	\[
	\bigcup_{(s_1,\ldots, s_m)\in [0,1]^m}\! \! \! \! \! \psi^1_{\kappa_1+s_1}\cdots \psi^m_{\kappa_m+s_m}(L).
	\]
	Since $H\in \mathcal{H}$, $C$ is a compact subset of $M$. Now define 
	\[
	K:=\{\rho(\mu)\in H_1(M;\R)\ |\ \mu \in \mathcal{M}(C;\phi_H) \ \text{with}\ \mathcal{A}_{H,\lambda}(\mu)+\langle c, \rho(\mu) \rangle =\sigma_{H:L}(c) \}.
	\]
	$K$ is again a convex compact subset of $H_1(M;\R)$, so it is characterized by its support function $\chi_K:H^1(M;\R)\to \R$. Given $h\in \partial_A\sigma_{H:L}(c)$ and $c'\in H^1(M;\R)$, an argument similar to that of the proof of Theorem \ref{thm1} shows the existence of a $\mu \in \mathcal{M}(C;\phi_H)$ satisfying 
	\begin{equation}
	\label{Mateq32}
	\langle c',h\rangle \leq \langle c',\rho(\mu) \rangle.
	\end{equation}
	Moreover, applying the statement of Proposition \ref{prop3}\footnote{More precisely we make use of the fact that (\ref{Mateq901}) holds.} together with the weak$^*$-convergence it follows that
	\[
	\sigma_{H:L}(c)=\mathcal{A}_{H,\lambda}(\mu)+\langle c,\rho(\mu) \rangle,
	\]
	which implies that $\rho(\mu)\in K$ and (\ref{Mateq32}) now implies $\langle c',h\rangle \leq \chi_K(c')$. Since $c'\in H^1(M;\R)$ was arbitrary this shows that $h\in K$, finishing the proof of Theorem \ref{thm2}. The proof of Corollary \ref{cor3} is identical to that of Corollary \ref{lem1}.
\end{proof}
\begin{rem}
\label{rem1}
Suppose each $a_k$ is $C^1$, $\alpha_{H:L}$ is $C^1$ and that $(a_k)_{k\in \N}$ $C^1$-converges to $\alpha_{H:L}$. In this setting one can avoid applying Jourani's result and "upgrade" the proof of Theorem \ref{thm1} to show the following: For every $c\in H^1(L;\R)$ there exists a measure $\mu \in \mathcal{M}(\phi_H)$ with rotation vector $\rho(\mu)=d\alpha_{H:L}(c)$, whose support satisfies
\[
\supp(\mu)\subset \overline{\bigcup_{t\in \R}\phi_H^t(L)}. 
\]
In particular, if $L$ is $\phi_H$-invariant then $\supp(\mu)\subset L$.
\end{rem}

\textbf{Acknowledgments}
I am very grateful to Vadim Kaloshin for teaching an excellent Nachdiplom lecture on Arnold Diffusion at the ETH Z{\"u}rich. Everything I know about Aubry-Mather theory I learned from him, and the insight that invariant measures should be "wild" in the presence of superconductivity channels was generously suggested to me by him. I am also very grateful to Leonid Polterovich. Several of the ideas (in particular the applications to twisted cotangent bundles) were generously suggested by him. I thank Paul Biran and Will Merry for helpful discussions and I am grateful to Nicolas Vichery, Vincent Humili{\`e}re and Claude Viterbo for their interest in the paper. Last (but certainly not least) I thank the extremely careful anonymous referee for lots of high-quality input. 

Part of this research was carried out during visits to Tel Aviv University and University of Maryland at College Park. I thank Leonid and Vadim for the opportunities to make these visits as well as both universities for truly excellent research atmospheres.

\bibliographystyle{plain}
\bibliography{BIB.bib}
\end{document}